\documentclass[12pt]{amsart}
\usepackage{amsmath,amssymb, amscd, graphicx}
\usepackage{youngtab}

\makeatletter
\@addtoreset{equation}{section}
\makeatother
\usepackage{enumerate}
\usepackage{color}

\newtheorem{theorem}{Theorem}[section]
\newtheorem{definition}[theorem]{Definition}

\newtheorem{proposition}[theorem]{Proposition}
\newtheorem{corollary}[theorem]{Corollary}
\newtheorem{remark}[theorem]{Remark}
\newtheorem{lemma}[theorem]{Lemma}

\def\book#1{\rm{#1}, }
\def\paper#1{\textit{#1}, }
\def\jour#1{\rm{#1}, }
\def\yr#1{({\rm{#1}) }}
\def\vol#1{\textbf{#1}}
\def\pages#1{\rm{#1}}

\def\publ#1{\rm{#1}, }
\def\by#1{{\rm{#1}, }}

\def\hN{{\widehat{N}}}

\def\hiota{{\widehat{\iota}}}
\def\hphi{{\widehat{\phi}}}

\def\tX{{{\widetilde{X}}}}
\def\Abl{{{\widetilde{w}}}}
\def\abl{{{{w}}}}

\def\tX{\widetilde{X}}
\def\Xrs{{X}}
\def\Rrs{{R}}
\def\Gammars{{\Gamma}}
\def\hR{\widehat{R}}

\def\fB{{\mathfrak{B}}}
\def\cK{{\mathcal{K}}}

\def\SL{\mathrm {SL}}

\def\wt{\mathrm {wt}}

\def\nuII{{\nu^{II}}}
\def\nuI{{\nu^{I}}}

\def\CC{{\mathbb C}}
\def\ZZ{{\mathbb Z}}
\def\NN{{\mathbb N}}
\def\QQ{{\mathbb Q}}

\def\WW{{{\mathcal W}}}
\def\cJ{\mathcal{J}}
\def\cO{{\mathcal{O}}}

\def\PP{{\mathbb P}}

\def\JJ{{\mathcal J}}

\def\cS{{\mathcal S}}

\def\nJJ{{{\mathcal J}^\circ}}
\def\nkappa{{{\kappa}^\circ}}
\def\nTheta{{{\Theta}^\circ}}
\def\nGamma{{{\Gamma}^\circ}}
\def\tv{{{\widetilde{v}}}}
\def\v{{{{v}}}}
\def\tw{{{\widetilde{w}}}}

\def\ker{\mathrm {ker }}
\def\Spec{\mathrm {Spec\ }}

\def\NMrs{{N}}

\def\NM{{N}}

\def\hr{{\widehat{r}{}}}
\def\hs{{\widehat{s}{}}}

\def\lms#1{{\lambda^{(s)}_{#1}}}
\def\lmr#1{{\lambda^{(r)}_{#1}}}
\def\lmsr#1{{\lambda^{(s+r)}_{#1}}}

\def\omegap#1{{\omega^{{\prime}}_{#1}}}
\def\omegapp#1{{\omega^{{\prime\prime}}_{#1}}}
\def\etap#1{{\eta^{{\prime}}_{#1}}}
\def\etapp#1{{\eta^{{\prime\prime}}_{#1}}}

\def\phiM#1{{\phi_{#1}}}
\def\phiH#1{{\widehat{\phi}_{#1}}}
\def\psig#1{{\psi_{#1}}}
\def\Psig#1{{\Psi_{#1}}}
\def\hpsig#1{{\widehat\psi_{#1}}}
\def\hPsig#1{{\widehat\Psi_{#1}}}

\def\muH#1{{\widehat\mu_{#1}}}
\def\alphaH{{\widehat\alpha}}

\def\hzeta{\widehat{\zeta}}

\theoremstyle{definition}

\theoremstyle{remark}

\numberwithin{equation}{section}

\begin{document}

% \title[short text for running head]{full title}
\title{The sigma function for trigonal cyclic curves}

%    Only \author and \address are required; other information is
%    optional.  Remove any unused author tags.

%    author one information
% \author[short version for running head]{name for top of paper}
\author{Jiryo Komeda}
\address{Department of Mathematics,\\
Center for Basic Education and Integrated Learning,\\
Kanagawa Institute of Technology,\\
Atsugi, 243-0292, JAPAN}
%\curraddr{}
\email{komeda@gen.kanagawa-it.ac.jp}

%    author two information
\author{Shigeki Matsutani}
\address{Industrial Mathematics,\\
National Institute of Technology, Sasebo College,\\
1-1 OkiShin-machi, Sasebo, Nagasaki, 857-1193, JAPAN\\}
%\curraddr{}
\email{smatsu@sasebo.ac.jp}
\thanks{The author (S.M.) thanks Professor Takeo Ohsawa for helpful comments
and discussions on the related topics.
J.K. and S.M. were supported by
the Grant-in-Aid for Scientific 
Research (C) of Japan Society for the Promotion
of Science, 
Grant No. 15K04830 and Grant No. 16K05187 respectively.}

\author{Emma Previato}
\address{Department of Mathematics and Statistics,\\
Boston University,\\
Boston, MA 02215-2411, U.S.A.}
%\curraddr{}
\email{ep@bu.edu}
%\thanks{}

%    \subjclass is required.
\subjclass[2010]{Primary14K25, 14H40. Secondary,
14H55, 14H50.}

\date{}

\dedicatory{}

%    "Communicated by" -- provide editor's name; required.
\commby{}

%    Abstract is required.
\begin{abstract}
A recent generalization of the ``Kleinian sigma function'' 
 involves the choice of a point $P$
of a Riemann surface $X$, namely a ``pointed curve'' 
$(X, P)$.
This paper concludes our explicit calculation 
of the sigma
function for curves
cyclic trigonal  at $P$.
We exhibit the Riemann constant for a 
Weierstrass semigroup at $P$
with minimal set of generators $\{ 3, 2r+s,2s+r\}$, 
 $r<s$, equivalently, non-symmetric; 
we construct a basis of $H^1(X, \mathbb{C})$ and
a fundamental 
2-differential on $X\times X$; we give the order of vanishing for
sigma on Wirtinger strata of the Jacobian of $X$, and a solution to the 
Jacobi inversion problem. 
\end{abstract}

\maketitle

\section{Introduction}\label{introduction}

The Weierstrass $\wp$ and $\sigma$ function, related by the identity 
$\wp(u)=-\dfrac{d^2}{du^2}\ln\sigma(u),$ are defined for elliptic curves.
Sigma in turn is related to Jacobi's theta functions \cite[20$\cdot$421]{ww}.
Theta functions were then defined by Riemann for any Abelian variety, while an
equivalent (in the sense of \cite{lang}, cf. Section \ref{transcendental}) function
was defined for hyperelliptic Jacobians by Klein \cite{klein1, K}. 
In fact, as pointed out in \cite{korotkin}, Klein defined it also
for all Jacobians of curves
of genus three \cite{klein2}, and in \textit{loc. cit.}, the authors
generalize it to any Jacobian by requiring a modular invariance
under the action of $\SL(2g, \mathbb{Z})$ (up to a root of unity).
A different approach, which we follow in this paper, originally 
proposed by Buchstaber,  Enolski\u{\i} and  Le\u{\i}kin \cite{bel}
and 
Eilbeck, Enolski\u{\i} and  Le\u{\i}kin \cite{EEL}
is based on Baker's results \cite{baker1907} that connect the (transcendental) sigma
function with the (algebraic) functions and differentials of the curve. In
this approach involves the choice of a point $P$ on the curve $X$, since
the relevant objects are written in terms of $H^0(X, \cO(*P))$. The
representation of the affine curve $X\backslash P$ is therefore also relevant,
and so is the Weierstrass semigroup (W-semigroup) at $P$.
Until recently, the only explicit formulas for sigma were produced for $(n,s)$
curves (cf. Section \ref{weierstrass}), which are plane affine curves that
have a smooth compactification by one point `at infinity', playing the role of
$P$, and 2-generator W-semigroup.
 Ayano \cite{ayano,ayano1} was 
able to follow this approach and construct sigma for
``telescopic'' (cf. Section \ref{weierstrass}) W-semigroups; at the same time,
the first two authors gave a construction for a new case, neither of $(n,s)$
nor of telescopic type \cite{komedamatsutani},   pursued in
\cite{KMP13, KMP16}. 
Those papers covered curves cyclic at $P$ with 
  W-semigroup that has minimal set of generators $ \{ 3, 2r+s,2s+r
\}$, where $r,s$ are
natural numbers with $r<s$.
Slow but steady (as in the ancient Zwahili
proverb ``Pole pole ndio mwendo''),  in this paper we complete the
explicit results for all curves trigonal cyclic at $P$ (in the sense of
\cite{accola}, cf. Subsection \ref{weierstrass}).
Our methods do not cover the non-cyclic case; for example,
Klein's quartic $X^3Y+Y^3Z+Z^3X$ in
projective coordinates $[X,Y,Z]$, the genus-three curve with maximum number of
automorphisms,  has W-semigroup $ \langle 3, 5,7
\rangle$ at $P=[0,0,1]$
but is not cyclic since all the
automorphisms of order three are conjugate and the quotient curves have genus
one. 
Indeed, our strategy,  introduced by Komeda and Matsutani
\cite{komedamatsutani}, 
is to piece together local coordinates at $P$ by using a suitable set of 
singular $(n,s)$ curves which are images of $X$: the images of
their affine rings generate the affine ring of $X\backslash P$.
These singular curves can be represented as monomial curves 
(cf. Subsection \ref{monomial}); in fact, singular ${Z}_3$ curves in the
  sense of \cite{farkaszemel}
because of the  action of 
$\mathbb{Z}/3\mathbb{Z}$. On the plus side, the value of our approach is 
that it works for any $k$-gonal cyclic cover of $\mathbb{P}^1$; although
we only
 work out the trigonal case because the general $k$ would necessitate
very large formulas, we provide general statements to the fullest extent
possible. We clarify that
 the words ``trigonal'' and $k$-gonal are used  only as an indication
 of the fact that the pointed curve $(X,\infty )$ has W-semigroup of type
 three, five, respectively; a different convention requires that $k$-gonal
 curves not be $j$-gonal for $j < k$ (as in trigonal curves, which by
 definition are not hyperelliptic);  this cannot be guaranteed in our
 examples, as noticed in  \cite[Ch. V, \S 70]{baker1897}.
Cyclic covers of $\mathbb{P}^1$ 
have been connected with vanishing properties of the theta function (for
recent work cf. \cite{accola, farkaszemel}), since there is an
induced action on the Jacobian.

We emphasize that the main effort in the theory is directed toward
explicitness. In the \cite{bel, ble, EEL} approach, 
the differentials that satisfy
the generalized Legendre relations (cf. Subsection \ref{sigma}) are built out
of algebraic functions on the curve, and this is where the cyclic action 
on  our local coordinates at $P$ is crucial.
The two versions of the Kleinian sigma function (modular vs. algebraic) have
not been fully compared yet; in fact, much remains to be done before the
algebraic approach can be exploited to its full extent in applications to
physics and dynamics.  This is important for two reasons: PDEs
and integrable systems can  be solved explicitly--for example, our formulas
for the order of vanishing of sigma on strata of the Jacobian give information
on the  qualitative behavior of
the solutions; and the numerical properties of the Weierstrass
points of the curve come into play, which opens vistas toward
applications to open questions in Weierstrass-semigroup theory
(cf. \cite[Section 2]{KMP13}).

We use the word ``curve'' for a compact Riemann surface: on occasion, we use a
singular representation of the curve; since there is a unique smooth curve
with the same field of meromorphic function, this should not cause confusion.
By ``natural number'' we mean a positive integer, excluding zero
from $\mathbb{N}$.
The contents  of this paper are organized as follows:
in Section \ref{weierstrass}, we collect definitions, properties, and
representation-theoretic interpretations 
of W-semigroups, cyclic covers
of $\mathbb{P}^1$, monomial curves.
Section \ref{cyclic}
 contains our main constructions  for the general 
trigonal cyclic case. %,  applied to the construction of a suitable basis of
		      %the first cohomology group
 %of the constant sheaf of the curve, and of the 
 %``normalized fundamental form'' \cite{bel}.
In Section \ref{transcendental},
 we introduce the transcendental aspect %objects (Jacobi variety, 
 %Abelian coordinates, Abelian vector fields, Riemann
 %theta function with characteristics, theta divisor, Riemann vector, 
 %and Wirtinger strata) and then 
 and the sigma function. Our explicit expressions
for the trigonal cyclic curves use classical and new 
theorems for the sigma function of
$(n,s)$ curves. In Section 
\ref{results}, we collect the
vanishing theorems for sigma 
and the solution to the Jacobi inversion problem. In the Appendix 
 we provide technical proofs for constructing our basis of
 $H^1(X,\mathbb{C})$, the first cohomology group
 of the constant sheaf of the curve.

%{\bf{
%Acknowledgments:}}
%The author (S.M.) thanks Professor Takeo Ohsawa for helpful comments
%and discussions on the related topics.
%J.K. and S.M. were supported by
%the Grant-in-Aid for Scientific 
%Research (C) of Japan Society for the Promotion
%of Science, 
%Grant No. 15K04830 and Grant No. 16K05187 respectively.

\section{Numerical and Weierstrass semigroups}\label{weierstrass}
We set up the notation for  W-semigroups,
and recall the results we use.

A pointed curve is a pair $(X,P)$, with $P$ a point of a curve $X$;
the W-semigroup for $X$ at $P$, which we denote by $H(X,P)$, 
 is the complement in $\mathbb{N}_0=\mathbb{N}\cup\{ 0\}$ 
of the Weierstrass gap sequence (W-gaps) $L$, namely  the set of natural
numbers 
$\{\ell_0 < \ell_1 < \cdots < \ell_{g-1} \}$ such that 
$H^0(\cK_X-\ell_iP)\ne H^0(\cK_X-(\ell_i-1)P)$, for
$\cK_X$ a  
representative of the canonical divisor (we identify
divisors and the corresponding sheaves; $\cK_X$ corresponds to 
the sheaf of holomorphic differentials $ \Omega_{X}$).
By the Riemann-Roch theorem, $H(X,P)$
 is a numerical semigroup.
In general, a numerical semigroup $H$
has a unique (finite) minimal set of generators,  $M(H)$, say,
and the cardinality $g$  of $\mathbb{N}_0 \backslash H$ is called
the genus of $H$
(for terminology, history and references cf. \cite{a-b}.)
 The Schubert index of the set $L$ is
\begin{equation}
\alpha(L) :=\{\alpha_0(L), \alpha_1(L), \ldots, \alpha_{g-1}(L)\},
\label{eq:alphaL}
\end{equation}
where 
$\alpha_i(L) := \ell_i - i -1$ \cite{EH}.
 For the smallest element $a_{\mathrm{min}}(H)$ in $M(H)$
we call a semigroup $H$ an $a_{\mathrm{min}}(H)$-semigroup,
therefore $H(X,P)$ for a trigonal cyclic pointed curve $(X,P)$ is a
3-semigroup.  

If we let the row lengths be
$\Lambda_i = \alpha_{g-i}+1$, $i\le g,$ then 
$\Lambda:
=(\Lambda_1, ..., \Lambda_g)$ is the Young diagram of the semigroup.
A W-semigroup is called symmetric when $2g-1$ occurs in the gap sequence;
thus, $H(X,P)$ is symmetric if and only if its Young diagram is symmetric,
i.e. invariant under reflection across the main diagonal.

{Telescopic numerical semigroups} were introduced in  Information Theory
(Goppa codes), because the parameters of the 
corresponding code, hence also the answer to the Diophantine Frobenius
Problem \cite[Ch. 7]{ramalfons},  can be computed efficiently.
They are defined to have generators 
$\{\alpha_0(L), \alpha_1(L), \ldots, \alpha_{g-1}(L)\}$
 that, up to reordering,
are a telescopic sequence:

\begin{definition}\label{tele1}
The sequence $\{ w_1,..., w_s\}$ is called telescopic if 
$\mathrm{gcd}(w_1,..., w_s)=1$ and
$w_i/d_i\in S_{i-1}$ $(i=2, ..., s)$,
 where $d_i=\mathrm{gcd} (w_1,...,w_i)$ and $S_i$ is the
semigroup generated by  $\{ w_1/d_i,..., w_i/d_i\}$, $1\le i\le s$.
\end{definition}

Telescopic semigroups are symmetric \cite[Lemma 7.1.7]{ramalfons},
but not \textit{vice versa}, as the generators $\alpha (L)=\{ 6,13,14,15,16\}$
show, since the Young diagram is symmetric \cite{KMP13} but the sequence is
not telescopic even after reordering.

Miura analyzed, in particular,
an affine model for the curves with W-semigroup at a point which is
telescopic, for applications to coding theory
\cite{Miu}.  His results appeared only  in Japanese,
and for that reason Ayano included the complete proofs in his paper
\cite{ayano} where he
 constructed the sigma function; he used sigma to obtain
 Jacobi inversion formulas and vanishing theorems over Wirtinger
 strata \cite{ayano1}. 

\subsection{The monomial Ring $B_H$}\label{monomial}\label{monomial}
A numerical semigroup $H$ has an associated ring
$B_H \simeq k[Z_1, Z_2,..., Z_\ell
]/ \ker\varphi$,  
where for a minimal set of generators 
$M=\{ m_1,...,m_\ell\}$, $\varphi$ is the
 epimorphism from the polynomial ring $\mathbb{C}[Z] := \mathbb{C}[Z_1,
   Z_2,..., Z_\ell]$  to
$B_H:=\CC[t^{m_1}, t^{m_2}, \ldots, t^{m_\ell}]$, given by
$ Z_a\mapsto t^{m_a}; 
$
Herzog shows that the kernel is a monomial ideal \cite[Prop. 1.4]{herzog}.

There is an action of $\mathbb{C}^\ast$
 on the monomial ring $B_{H}$, given as
$Z_j\mapsto a^{m_j} Z_j$ for $a\in \mathbb{C}^\ast$. Accordingly,
we say that $Z_a$ has weight $a$.

\subsection{Weierstrass normal form}
We now consider the ``Weierstrass normal form'' (W-normal form): this 
gives a (possibly singular) model for any
pointed  curve in affine space of the same dimension as a minimal
set of generators of the W-semigroup. 
Baker \cite[Ch. V, \S\S 60-79]{baker1897} 
gives a complete review, proof and examples of the theory 
(he calls it ``Weierstrass canonical form''), which is a
generalization of Weierstrass' equation for elliptic curves.
We refer to Kato \cite{kato}, who 
also produces this representation, with proof that it exists.

\begin{proposition}
A pointed curve $(X,\infty )$ 
with W-semigroup $H(X,\infty)$
for which  $a_{\mathrm{min}}(H(X,\infty))=m$
 can be viewed as an $m:1$  cover of $\mathbb{P}^1$; we denote by
$x$  an affine coordinate on $\mathbb{P}^1$ such that
the point $\infty$ on $X$ is mapped to $x=\infty$;
we let
$m_i :=\min\{h \in H(X,\infty)\setminus \{0\}\ 
|\ i \equiv h \ \mbox{ mod } m\}$,
$i =0, 1, 2, \ldots, m-1$,
$m_0 = m$ and  $n=\min\{m_j \ | \ (m, j) = 1\}$. Then,
$(X,\infty)$
 is defined by an irreducible equation,
\begin{equation}
f(x,y)= 0,
\label{eq:WNF1a}
\end{equation}
for a polynomial $f\in \CC[x,y_n]$ of type,
\begin{equation}
f(x,y):=y^m + A_{1}(x) y^{m-1}
+ A_{2}(x) y^{m-2}
+\cdots + A_{m-1}(x) y
 + A_{m}(x),
\label{eq:WNF1b}
\end{equation}
where the $A_i(x)$'s are polynomials in $x$,
of degree $\leq jn/m$, with equality being attained only for
$j=m$:
$$
A_i = \sum_{j=0}^{\lfloor i n/m\rfloor} \lambda_{i, j} x^j,
\  \lambda_{i,j} \in \CC,\ \lambda_{m,m}=1.
$$
\end{proposition}
The affine curve Spec $\CC[x,y]/f(x,y)$ may be
singular and we denote by $X$ its unique normalization.
Kato's proof shows that the
affine ring of the curve $X\backslash\infty$ is generated by functions
$y_{m_i}$ whose only pole is
 $\infty$ with order  the $g$ non-gaps $m_i$:
let $I_m :=\{m_1, m_2, \ldots, m_{m-1}\}\setminus \{m_{i_0}\}$,
where $i_0$ is such that $n = m_{i_0}$, take 
 $x=y_{m}$ and $y=y_{n}$; then the affine ring 
of $X\backslash\infty$ can be presented as
$\CC[x,y_n, y_{m_3}, \ldots, y_{m_\ell}]$
for $i_j \in M_g$, where
$M_g :=\{ m_1, m_2, \ldots, m_\ell\} \subset \NN^{\ell}$ 
with $(m_i, m_j) = 1$ for $i\neq j$, $m_1=m, m_2=n$, is a minimal set of
generators for $H(X,\infty)$.

The curve in W-normal form, namely (\ref{eq:WNF1b}),
 admits a local $\mathbb{Z}/m\mathbb{Z}$-action at $\infty$,
 %where $\mathbb{C}^\ast$ is the multiplicative group of the complex numbers,
 in the following sense.
Sending $Z_1$ to $1/x$ and $Z_{i}$ to $1/y_{m_i}$ 
for $m_i \in \alpha (L), \ i>1$, we have the same kernel 
as under the homomorphism to the semigroup ring.
The action is defined by sending $Z_i$ to $\zeta_m^iZ_i,$ 
where $\zeta_m$  is a primitive $m$-th root of unity. 
%For the cyclic Weierstrass canonical form, this action is global on
%$X$. 

Our  results on the sigma function pertain to the case in which
$A_1, \ldots, A_{m-1}$ vanish, i.e.,
$y_n^m = A_1(x)$; 
this is the case if and only if $(x,y)\mapsto x$
 is a cyclic cover with Galois group $\mathbb{Z}/m\mathbb{Z}$.
We call this case ``cyclic W-normal form'', and
in this paper, we obtain vanishing and inversion theorems for $m=3$.

We treat as distinct the case when the curve has W-semigroup $H(X,\infty)$
generated by two elements $(m,n)$
%, $(3,s)$ in our case (for $r=0$, say) 
and
the W-normal form defines a non-singular plane affine curve.
This is known as  $(n,s)$-curve 
\cite{bel, EEL} and much work on the sigma function  already 
appeared, both in the cyclic and non-cyclic case. Our results for the
cyclic case go through for $(3,s)$ curves (for $r=0$)
although some statements need to be
slightly modified and will be specified in the rest of the paper.

\section{Trigonal cyclic curves}\label{cyclic}

\subsection{3-semigroups}
\label{sec:378}

The minimal set of generators of a  3-semigroup can have
two or three elements because 3 and the smallest element
that is not a multiple of 3 generate all larger numbers with the same
residue mod 3. We record the easy:   
\begin{proposition} \label{prop:3hrhs}
A numerical semigroup $\langle3, p, q\rangle$ 
which is not generated by two elements, has minimal set of generators
(not necessarily in increasing order) 
of type $\{ 3, 2r+s, r+2s\}$ for some 
positive integers
$r, s, s> r$, and has 
genus $g=r+s-1$.
\end{proposition}

\begin{proof}
Of the three possibilities, (1) $p+q = 0$ mod 3, 
(2) $p+q=1$ mod 3, (3) $p+q=2$ mod 3, (2) would imply that both $p$
and $q$ are 2 mod 3, whereas the set of generators is assumed minimal.
We similarly exclude (3), so we can write 
$p = 3\ell + a$, $q= 3\ell' + a'$, where $a\cdot a' = 2$, so that 
 $2a-a'=0$ mod 3; finally,
$r=(2p-q)/3$ and $s=(2q-p)/3$ are positive integers
because if $q> 2p-1$, again 3 and $p$ generate. 
The genus is the number of gaps, $r+s-1$.
\end{proof}

\begin{remark}\label{rmk:3rs_3s} For $r=0$
in Proposition \ref{prop:3hrhs}
$\langle 3, 2r+s, r+2s\rangle$ is reduced to
the case of two generators, 
$\langle3,s\rangle$
and $X$ is  the compactification of  a plane smooth curve:
it is the $r=3$ case 
of a cyclic $(n,s)=(r,s)$ curve in \cite{MP08, MP14}; 
in this paper we generalize those
results. 
\end{remark}

For brevity, henceforth we
let 
$\hr:={2r+s}$ and
$\hs={2s+r},
$
so that for $r < s$, $\hr < \hs$.
Unless $r=0$, we assume that $3,\hr, \hs$ are a minimal set of generators of
the W-semigroup. 
 %In general, 3, $\hr$ and $\hs$ are not necessarily coprime but
 %as mentioned in Proposition \ref{prop:3hrhs}
 %Examples are given  in Table 1.
 When $s>r>0,$ the triple $(3, \hr, \hs)$ is not telescopic 
in any order (because as we demonstrate,  the corresponding 
Young diagram is not symmetric)
 whereas a semigroup with two
generators  is clearly telescopic.

 %\subsection{The monomial ring $B_H$}

We give more detail on the monomial ring (defined in Subsection
\ref{monomial}) 
of a 3-semigroup. Again by \cite[Th. 3.7]{herzog},

\begin{proposition} \label{prop:Z4}
For the $\mathbb{C}$-algebra homomorphism,
$$
	\varphi : 
k[Z] := k[Z_3, Z_{{\hr}}, Z_{{\hs}}] \to k[t^a]_{a\in \{ 3, \hr, \hs\}},\
\ Z_a\mapsto t^a,
$$
the kernel of $\varphi_{g}$ is generated by
$f_{b} = 0,$ $b = {2\hr}, {\hs+\hr}, {g},$ where
\begin{equation}
f_{{2\hr}} = Z_{{\hr}}^2 - Z_3^{{r}} Z_{{\hs}}, \quad
f_{{\hs+\hr}} = Z_{{\hr}} Z_{{\hs}} - Z_3^{s+r}, \quad
f_{{g}} = Z_{{\hs}}^2 - Z_3^{{s}} Z_{{\hr}}. \quad
\label{eq:rel}
\end{equation}
\end{proposition}

%\begin{proof} This is due to Herzog \cite[Th. 3.7]{herzog}.
%\end{proof}
%The relations in (\ref{eq:rel})
%are given by the $2 \times 2$ minors of
%\begin{equation}
%\displaystyle{
%\left|\begin{matrix}
% Z_3^{{r}} & Z_{{\hr}} & Z_{{\hs}} \\
% Z_{{\hr}} & Z_{{\hs}}  & Z_3^{{s}}\\
%\end{matrix} \right|}.
%\label{eq:2x2minor4Z}
%\end{equation}

%We say that $Z_a$ has weight $a$, consistent with the monomial-ring action of
%$\mathbb{C}^\ast$ defined in Subsection \ref{monomial}. 

\subsection{Singular trigonal  cyclic curves}\label{singular}

Let
\begin{eqnarray*}
k_{{s}}(x) &:=& (x-b_1) \cdots (x-b_{{s}})
\equiv x^s + \sum_{i=1}^s \lms{i} x^{s-i}, \\
k_{{r}}(x) &:=& (x-b_{s+1}) \cdots (x-b_{s+r})
\equiv x^r + \sum_{i=1}^r \lmr{i} x^{r-i}, \\
k_{s+r}(x) &:=& (x-b_1) \cdots (x-b_{s}) \cdots (x-b_{s+r})
\equiv x^{r+s} + \sum_{i=1}^{r+s} \lmsr{i} x^{s+r-i}, \\
\end{eqnarray*}
and $\lmr{0}= \lms{0}= \lmsr{0}=1$.
We define $f_{{2\hr}}, f_{{\hs+\hr}}, f_{{2\hs}} \in \CC[x, y_{{\hr}}, y_{{\hs}}]$ by
$$
f_{{2\hr}} = y_{{\hr}}^2 - y_{{\hs}} k_{{r}}(x), \quad
f_{{\hs+\hr}} = y_{{\hr}} y_{{\hs}} - k_{{r}}(x) k_{{s}}(x), \quad
f_{{2\hs}} = y_{{\hs}}^2 - y_{{\hr}} k_{{s}}(x),
$$ 
 %obtained from  the $2 \times 2$ minors of
%\begin{equation}
%\displaystyle{
%\left|\begin{matrix}
% k_{{r}}(x)  & y_{{{\hr}}} & y_{{{\hs}}} \\
% y_{{{\hr}}} & y_{{{\hs}}}  & k_{{s}}(x)\\
%\end{matrix} \right|}, \quad
%\label{eq:2x2minor4}
 %\end{equation}
 cf. \ref{eq:rel}. %(\ref{eq:2x2minor4Z}).
 We obtain cyclic W-normal forms,
\begin{equation}
y_{\hr}^3 = k_r^2(x) k_s(x), \quad
y_{\hs}^3 = k_s^2(x) k_r(x). \quad
\label{eq:y^3=f(x)}
\end{equation}
 
The monomial-ring action trivially extends to  the ring
$$
R_\lambda:=\QQ[x, y_\hr, y_\hs, \lms{1}, \ldots, \lms{s},
\lmr{1}, \ldots, \lmr{r} ]/(f_{{2\hr}}, f_{{\hs+\hr}}, f_{{2\hs}}),
$$
graded
by 
$\wt_\lambda:R_\lambda \to \ZZ ,$ 
$\lms{i}\rightarrow 3i$ and $\lmr{i}\rightarrow 3i$,
so that the equations that define the curve are homogeneous.  This ring
parametrizes the moduli (in a loose sense) of the curves over the rationals,
and it can be used to address number-theoretic issues, although we do not do
so in this paper.

We consider the curve $X^{(3,\hr,\hs)} =\Spec R^{(3,\hr,\hs)}$, where 
$$
R^{(3,\hr,\hs)} 
:=\CC[x, y_{{\hr}}, y_{{\hs}}]/ (f_{{2\hr}}, f_{{\hs+\hr}}, f_{{2\hs}}),
$$
so that $X^{(3,\hr,\hs)}$ is a curve in affine 3-space with coordinates
$(x,y_{{\hr}},y_{{\hs}})$. For brevity, we
 write $R$ for $R^{(3,\hr,\hs)}$
and $X$, a $(3, \hr, \hs)$-type curve, for $X^{(3,\hr,\hs)}$.
 %\subsection{$\mathbb{Z}/3\mathbb{Z}$-action on $X$}
%\label{subsec:GmX}
Equation
 (\ref{eq:y^3=f(x)}), shows  that 
there is a global $\mathbb{Z}/3\mathbb{Z}$ action on $\Xrs$ given by:
$$
\hzeta_3(x, y_{{\hr}}, y_{{\hs}})= (x, \zeta_3 y_{{\hr}}, \zeta_3^2
y_{{\hs}}), 
$$
where $\zeta_3$ is a primitive 3rd root of unity.

We denote the branch points of $X$ viewed as a cover of the $x$-line by
$$
B_i:=(x=b_i, y_{\hr}=0, y_{\hs}=0),\quad
(i=1,2,3,\ldots, s+r),
$$
and we consider the divisor,
\begin{equation}
\fB_0:=B_1 + B_2 +\cdots +B_{r+s}.
\label{eq:B0}
\end{equation}

When $r=0$, 
(\ref{eq:y^3=f(x)}) corresponds to the plane curves
of genus $g=s-1$,
$$
y_{\hr}^3 = k_s(x), \quad y_{\hs} = y_{\hr}^2, \quad
$$
and our results clearly hold for that case also; assume henceforth $r> 0$.

%\subsection{Non-singularity of $X$}

 Nagata's Jacobian criterion can be used as in \cite[Prop.3.2]{KMP13}
 %\cite[Theorem 30.10]{Mat} 
 to show that
$\Spec \Rrs$ is non-singular, by checking that 
the $3\times 3$
matrix of the partial derivatives in  $(x, y_{{\hr}}, y_{{\hs}})$ 
of $f_{{2\hr}}, f_{{\hs+\hr}}$ and $f_{{2\hs}}$ 
has rank 2 at each point of the curve.

\subsection{The semigroup sequence}

We define 
monic monomials $\phiM{i}$
in the ring $R$,
whose poles at $\infty$ are the elements
of the W-semigroup. The $\phiM{i}$'s are determined by 
 requiring that 
the ordered set
$(R_\phi:=\{\phiM{n}\ |\ n \in \mathbb{N}\},<)$,
satisfies $N(n)<N(n+1)$, where
$\NMrs(n) :=N^{(3,\hr,\hs)}(n) = \wt(\phiM{n})$ (the order of pole at
$\infty$), and
$$R=R^{(3,\hr,\hs)} =
H^0(\Xrs\backslash\infty,\cO_{\Xrs})
=\bigoplus_{n=0} \CC \phiM{n}
=\bigoplus_{\eta \in R_\phi} \CC \eta.
$$

\subsection{Differentials of the first kind}
Using (\ref{eq:B0}), define a subspace 
of $\Rrs$: $
\widehat R:=
\{f\in R \ |\ \mbox{there exists }\ell 
\mbox{ such that }(f)-\fB_0 +\ell \infty > 0 \}
=\oplus_{n=0}^\infty \CC\phiH{n},
$
with basis  an ordered (by weight) subset 
$(\widehat R_{\phiH{}}:=\{\phiH{i}\}_{i=0,\ldots}, <)$ of
$(R_{\phi},<)$.
Let 
$\widehat R_{\phiH,n} :=\{\phiH{0}, \phiH{1}, \ldots, \phiH{n}\}$.
We denote
 the weight  of $\phiH{n}$ by
$
	\hN(n) : = \hN^{(3,\hr,\hs)}(n) : = \wt(\phiH{n}),
$
  consistent with the $\mathbb{Z}/3\mathbb{Z}$-action.

Since $\wt(y_{\hr}y_{\hs})=3(r+s)$ and
$\wt(x^i) = 3i$ $(i=0,1,\ldots,g-1)$ is less than $3(r+s)$,
we obtain the following statement:

\begin{lemma} \label{lmm:phiH}
\begin{enumerate}
\item
$\phiH{i}$ $(i=0,1,2, \ldots)$ is divisible by $y_\hr$ or $y_\hs$.
\item
$\phiH{g+2} =y_{\hr}y_{\hs}$
\item
For $i<g+2$, there is a non-negative integer $a$ 
such that $\phiH{i} =x^a y_{\hr}$ or $\phiH{i} =x^a y_{\hs}$.
\item
$\phiH{g-1} \phiH{g} =x^{g-1} y_{\hs}y_{\hr}$.
\item 
$\phiH{2} /\phiH{1} =x$ if $\hr+3<\hs$ or $s-r>3$.
\item
By letting
$g_\hr:=\displaystyle{\left\lfloor \frac{\hs-1}{3}\right\rfloor}$ and
$g_\hs:=\displaystyle{\left\lfloor \frac{\hr-1}{3}\right\rfloor}$,
we have $g=g_\hr+g_\hs$ and
$\displaystyle{
\widehat R_{\hphi,g-1} = 
}$
$\displaystyle{
\widehat R_{\hr} \bigoplus \widehat R_{\hs},
}$
where
$$
\widehat R_{\hr}:=\{x^i y_\hr \ | \ i = 0, \ldots, g_\hr-1 \}, \quad
\widehat R_{\hs}:=\{x^i y_\hs \ | \ i = 0, \ldots, g_\hs-1 \}.
$$
\end{enumerate}
\end{lemma}

\begin{proof}
(1) and (5) are obvious. %We consider (2), (3) and (6):
It is obvious that $y_\hr y_\hs$ belongs to $\widehat R$, so
 $\phiH{q}= y_\hr y_\hs$ for $q$ such that
$\widehat N(q) = 3r + 3s$.
Since $y_\hr y_\hs = k_r(x)k_s(x)$,
$y_\hr y_\hs/(x-b_i)$ belongs to $R$,
but not to $\widehat R$.
We let $i_r$ and $i_s$ such that
$\phiH (i_r) = y_\hr$ and $\phiH (i_s) = y_\hs$,
which implies $i_r < i_s < q-2$.
Hence clearly,
for  $i < q$, 
$\displaystyle{
\widehat \phi_i =
\left\{\begin{matrix}
 x^{j_i} y_\hr\\
 x^{k_i} y_\hs\\
\end{matrix}\right.
}$,
where $j_i$ and $k_i$ are non-negative integers.
Hence $\hN(q-a)=\hN(q)-a$ for $a =1,2$ and $\hN(q-3)=\hN(q)-4$.
If $i>q$, it is obvious that $\hN(i) = \hN(i-1)+1$.
Noting that $\hN(0) = \hr=2r+s$, 
$q= g_\hr + g_\hs +2=r+s+1 = g+2$,
(6) is proved and thus (3) and (2) are also proved.
Hence $\hN(g-1)=2g-2 + r + s$, and
$\phiH{g-1}\phiH{g}$ is equal to 
$x^{g_\hr-1}y_\hr\cdot x^{g_\hs}y_\hs$
or $x^{g_\hs-1}y_\hs\cdot x^{g_\hr}y_\hr$ and both
are equal to $x^{g-1} y_{\hs}y_{\hr}$. This proves (4).
\end{proof}

In the course of the proof we showed the following:
\begin{lemma}\label{lmm:Ng-1}
The following relations hold:
$
\widehat N(g-1) = 2g- 2 + r +s
$ and
$
\widehat N(i) = g + r +s + i , \quad i \ge g.
$
\end{lemma}
Next we observe, by checking that the finite poles of each term of the sum
have different orders:
\begin{lemma}
For a non-negative integer $n<s+r$, if the differential
$$
\sum_{i = 0}^n a_i \frac{x^i d x}{ y_{{\hr}} y_{{\hs}}}
\equiv\sum_{i = 0}^n a_i \frac{x^i d x}{ k_{{r}}(x) k_{{s}}(x)},
$$
is holomorphic over $\Xrs$, then each
$a_i$ vanishes.
\end{lemma}

\begin{proposition}\label{def:holo1-form}
%There is a subsequence $\{\phiH{i} \}$ of $\{\phiM{i} \}$ so that
$
\displaystyle{\left\{\frac{\phiH{i} d x}{ y_{{\hr}} y_{{\hs}}}
\ \Bigr|\ \phiH{i} \in \widehat R_{\phiH{}}\right\}}
$
is a basis of
$H^0(\Xrs\setminus \infty, \Omega_{\Xrs})$ as a
$\CC$-vector space,
in particular a basis of  $H^0(\Xrs, \Omega_{\Xrs})$ 
is given by 
$
\displaystyle{\left\{\nuI_{i} := \frac{\phiH{i-1} dx}{3y_{{\hr}} y_{{\hs}}},
  \quad \Bigr|\ \phiH{i} \in \widehat R_{\phiH{}},\
i=1,2,\ldots,g
\right\}}.
$
\end{proposition}

\begin{corollary} Coordinates for
the canonical embedding of $X$ into $\PP^{g-1}$ can be given by 
$[\phiH{0},\phiH{1}, \ldots , \phiH{g-2},\phiH{g-1}]$.
\end{corollary}

\begin{remark}
\label{rmk:3.8}
{\rm{
When $r=0$, $\widehat R$ is given by
$\widehat R = y_\hr R$,
 i.e., $\widehat \phi_i =y_\hr\phi_i$ for every 
$i=0,1,2,\ldots,$ 
so that
$\displaystyle{
\nuI_i=
\frac{\phi_{i-1}dx}{3 y_{\hr}^2}
=
\frac{\phiH{i-1}dx}{3 y_{\hr}y_{\hs}}.
}$
}}
\end{remark}

\subsection{The canonical divisor}
\label{subsec:KX}

\begin{remark}\label{rmk:CanonDiv}{\rm{
The divisors of our basis of one-forms are given by:

\begin{gather*}
\begin{split}
(\nuI_{1})&=B_{s+1}+\cdots +B_{s+r}+(2s+r)\infty, \\
(\nuI_{i})&=B_1+B_2+\cdots +B_{s}+(2r+s)\infty, \ 1\le i <s\\
(\nuI_{j})&=\sum_{a=0}^2
(0,\zeta_3^a\root 3\of{k_{{\hr}}(0)},\zeta_3^{2a}\root 3\of{k_{{\hs}}(0)})
+B_{s+1}+\cdots +B_{s+r} +(2s+r-3)\infty, \ 1\le j<r,\\
\end{split}
\end{gather*}
where $B_a :=(b_a, 0, 0)$ $(a=1,2,\ldots,s+r)$.
Noting $(s+r)=g+1$,
 we have:
 \begin{gather*}
\begin{split}
 \cK_{\Xrs} &\sim 2 (g-1) \infty - 2(B_{s+1}+\cdots + B_{s+r}
              - ({r})\infty))\\
  &\sim 2 ((g-1) \infty -2 (B_1+B_2+\cdots +B_{{s}} - ({s})\infty))\\
  &\sim 2(g-1) \infty -
(B_1+B_2+\cdots +B_{{s}}+B_{s+1}+\cdots + B_{s+r} - (s+r)\infty).\\
\end{split}
\end{gather*}
 %deduced by using
%\begin{gather*}
%\begin{split}
%  B_1+B_2+\cdots + B_{s+r} - (s+r)\infty &
%\sim -(B_{s+1}+\cdots + B_{s+r} -({r})\infty)\\
%&\sim 2(B_{s+1}+\cdots + B_{s+r} -({r})\infty)\\
 %  &\sim -(B_1+B_2 +\cdots +B_{{s}} -({s})\infty)\\
 % & \sim 2(B_1+B_2 +\cdots +B_{{s}} -({s})\infty).\\
%\end{split}
%\end{gather*}

In fact, any positive divisor linearly equivalent to $\cK_{\Xrs}$
must include points of $\Xrs \setminus \infty$,
because $H(X,\infty )$ is non-symmetric, whereas in the symmetric case,
which includes $(n,s)$-curves,
$(2g-2)\infty$ is a canonical divisor.
}}
\end{remark}
\subsection{Differentials of the second and  third kind}
\label{differential forms 4}
We produce 
 a differential form which, up to a tensor of holomorphic one-forms,
is  the  normalized fundamental differential of the
second kind in \cite[Corollary 2.6]{fay},
namely, a two-form $\Omega(P_1, P_2)$ on $\Xrs\times \Xrs$
which is symmetric,
\begin{equation}
\Omega(P_1, P_2)=\Omega(P_2, P_1),
\label{eq3.1.6}
\end{equation}
has its only pole (of second order) along the diagonal of  $\Xrs\times \Xrs$,
and in the vicinity of each point $(P_1,P_2)$
 is expanded in power series as
\begin{equation}
\Omega(P_1, P_2)=\Big(\frac{1}{(t_{P_1} -t_{P_2} ')^2  } +d_\ge(1)\Big)
   d t_{P_1} \otimes d t_{P_2}
\ \ (\hbox{\rm as}\  P_1\rightarrow P_2).
\label{expansion}
\end{equation}
We follow the work done in \cite{bel, ble, EEL} for $(n,s)$ curves,
adapted in \cite{komedamatsutani, KMP13}  to
 $\langle3,4,5\rangle$- and $\langle3,7,8\rangle$-curves.
The explicit form of $\Omega$ enables
 the algebraic (as opposed to modular) construction of sigma:
we call it  {\lq\lq}EEL-construction\rq\rq, as it is given theoretically 
 in \cite{EEL}. A computation shows:

\begin{proposition}\label{prop:Sigma}
Let $\Sigma\big(P, Q\big)$ be the following form,
\begin{equation}
   \Sigma\big(P, Q\big)
   :=\frac{ y_{{{\hr}},P} y_{{{\hs}},P} +y_{{{\hr}},P} y_{{{\hs}},Q} +y_{{{\hr}},Q} y_{{{\hs}},P}}
{3(x_P - x_Q)  y_{{{\hr}},P} y_{{{\hs}},P} } d x_P .
\label{eq:Sigma}
\end{equation}
Then $\Sigma(P, Q)$ has the properties
\begin{enumerate}
\item $\Sigma(P, Q)$ as a function of $P$ is singular at
$Q=(x_Q, y_{{{\hr}},Q}, y_{{{\hs}},Q})$ and  $\infty$, and
vanishes at $\hzeta_3(Q)= (x_Q, \zeta_3 y_{{{\hr}},Q}, \zeta_3^2 y_{{{\hs}},Q})$.

\item $\Sigma(P, Q)$ as a function of $Q$ is singular at $P$ and
$\infty$.
\end{enumerate}
\end{proposition}

%\begin{proof}
%Direct computation.
%\end{proof}

\begin{remark}\label{rmk:Sigma}
\begin{enumerate}
\item When $r=0$ 
(\ref{eq:Sigma}) reduces to

$$\displaystyle{
   \Sigma\big(P, Q\big)
   =\frac{ y_{{{\hr}},P}^2  +y_{{{\hr}},P} y_{{{\hr}},Q} 
+y_{{{\hr}},Q}^2}
{3(x_P - x_Q)  y_{{{\hr}},P}^2  } d x_P.
}$$
 %which is identical to that of a cyclic trigonal
 %plane curve in \cite{MP14}.
 \item

The Galois group $\mathbb{Z}/3\mathbb{Z}$
of the covering  $X\to (\PP^1)$, $P
\mapsto x$, acts on
the numerator of $\Sigma$ in Proposition \ref{prop:Sigma}. This is
what enables our technique for $k$-cyclic pointed curves $(X,P)$;
finding the explicit expression in  affine coordinates
for the normalized fundamental differential of the
second kind is the focus of current research even for  $(n,s)$
curves (2-generator W-semigroup).
 %mentioned in \cite{EEL} and is investigated in \cite{S} recently.
\end{enumerate}
\end{remark}

\begin{definition} \label{def:Drs}
\begin{enumerate}
\item
For $f\in \CC[x_P,x_Q, y_{\hs,P} y_{\hr,P},
y_{\hs,Q} y_{\hr,Q}]$,
$\widehat\pi_{Q,i}(f)$ is the 
$y_{\hr,Q} \phiH{i}(Q) $-coefficient of $f$ 
if $\phiH{i}\in \widehat R$
or the $y_{\hs,Q} \phiH{i}(Q) $-coefficient
 if $\phiH{i} \in \widehat R$.

\item

$\widehat D_{r,s}(P, Q) \in 
\ZZ[x_P,x_Q, 
y_{\hs,P},y_{\hr,Q},
y_{\hs,Q},y_{\hr,P},
\lmr{1},\ldots,\lmr{r}, \lms{1},\ldots,\lms{s}]$ is defined as:
$\displaystyle{
\widehat D_{r,s}(P, Q) : =
y_{\hs,P} y_{\hr,Q} D^{(+)}_{ s,r}(P, Q) +
y_{\hs,Q} y_{\hr,P} D^{(-)}_{ s,r}(P, Q),
}$
where 

$$
D^{(\pm)}_{s,r}(P, Q)
 \in \ZZ[x_P,x_Q, \lmr{1},\ldots,\lmr{r},\lms{1},\ldots,\lms{s}],
$$
 %\in 
%\ZZ[x_P,x_Q, \lmr{1},\ldots,\lmr{r}, \lms{1},\ldots,\lms{s}]$, $D_{r,s}(P, Q)
%:=$ 
 \begin{equation*}
\begin{split}
D^{(+)}_{ s,r}(P, Q)&:=
\sum_{j=0}^{s+r-2}
 \sum_{i=0}^{s+r-j-2}
(i+1)\lmsr{j}x_P^{r+s-j-i-2} x_Q^{i}\\
&+ \sum_{j=0}^{s-2}
 \sum_{i=0}^{s-j-2}
 \sum_{k=0}^{r}
(i+1)\lms{j}\lmr{r-k} x_P^{s-j-i-2} x_Q^{k+i},\\
D^{(-)}_{ s,r}(P, Q)&:=\sum_{j=0}^{s+r-2}
 \sum_{i=0}^{s+r-j-2}
(i+1)\lmsr{j}x_Q^{r+s-j-i-2} x_P^{i}\\
&+ \sum_{j=0}^{r-2}
 \sum_{i=0}^{r-j-2}
 \sum_{k=0}^{s}
(i+1)\lmr{j}\lms{s-k} x_Q^{r-j-i-2} x_P^{k+i}.
\end{split}
\end{equation*}
\end{enumerate}
\end{definition}
Note that
$\widehat D_{r,s}(P, Q)$ is homogeneous with respect to the extended weight 
$\wt_\lambda$.

On the non-singular curve $\Xrs$, the following Proposition
holds: the proof is given in
Appendix A. The Corollary, which gives the expression for the normalized
fundamental differential, is a straightforward verification.
\begin{proposition} \label{prop:dSigma}
There exist  differentials $\nuII_{j}=\nuII_{j}(x,y)$ $(j=1, 2, \cdots, g)$
of the second kind such that
they have
 a simple pole at $\infty$ and satisfy the relation,
\begin{equation}
\begin{split}
  &d_{Q} \Sigma\big(P, Q\big) -
  d_{P} \Sigma\big(Q, P\big)\\
   &\quad\quad=
     \sum_{i = 1}^{g(\hr,\hs)} \Bigr(
         \nuI{i}(Q)\otimes \nuII_{i}(P)
        - \nuI{i}(P)\otimes \nuII_{i}(Q)
     \Bigr)
   \label{eq3.4},
\end{split}
\end{equation}
where
\begin{equation}
 d_{Q} \Sigma\big(P, Q\big)
   :=d x_P \otimes d x_Q\frac{\partial }{ \partial x_Q}
   \frac{
y_{{{\hr}},P} y_{{{\hs}},P}
+y_{{{\hr}},P} y_{{{\hs}},Q}
+y_{{{\hr}},Q} y_{{{\hs}},P}}
{(x_P - x_Q) 3 y_{{{\hr}},P}  y_{{{\hs}},P} } d x_P .
\end{equation}
The set of differentials $\{\nuII_{1}$, $\nuII_{2}$,
$\nuII_{3}$, $\cdots$, $\nuII_{g}\}$
 is
determined modulo the linear space spanned by
$\langle\nuI_{j}\rangle_{j=1, \ldots, g}$ and it has representatives
\begin{equation*}
\begin{split}
\nuII_{i}(P) =
\left\{
\begin{matrix}
\displaystyle{\frac{ (\widehat\pi_i(\widehat D_{r,s}(P,Q))/y_{\hr,P}) dx_P }
                   {3 y_{\hs,P}}}
& \mbox{ if }  \phiH{i} \in \hR_{\hr}, \\
\displaystyle{\frac{ (\widehat\pi_i(\widehat D_{r,s}(P,Q))/y_{\hs,P})) dx_P }
                   {3 y_{\hr,P}}}
& \mbox{ if }  \phiH{i} \in \hR_{\hs},\\
\end{matrix}
\right.
\end{split}
\end{equation*}
where
$\hR_{\hr}$ and $\hR_{\hs}$ are defined in Definition \ref{def:holo1-form}.
\end{proposition}

%%%%%%%%%%%%%%%%%%%%%%%%%%%%%%%%%%%%%%%%%%%%%%%%%%%%%%%%%%%%%%%%%%%%%%%%%%%%%%

\begin{corollary}
\label{cor:Sigma}
\begin{enumerate}
\item
The one-form
$\displaystyle{
\Pi_{P_1}^{P_2}(P):= \Sigma(P, P_1)dx -  \Sigma(P, P_2)dx
}$
is a differential of the third kind,  whose only
(first-order) poles are
$P=P_1$ and $P=P_2$, with residues $+1$ and $-1$
respectively.

\item
The fundamental differential of the second kind
 $\Omega(P_1, P_2)$ is given by
\begin{equation}
\begin{split}
\Omega(P_1, P_2) &= d_{P_2} \Sigma(P_1, P_2)
     +\sum_{i = 1}^g \nuI_{i}(P_1)\otimes \nuII_{i}(P_2)\\
  &=\frac{F(P_1, P_2)dx_1 \otimes dx_2}
{9(x_1 - x_2)^2
y_{{{\hr}},P_1}
y_{{{\hs}},P_1}
y_{{{\hr}},P_2}
y_{{{\hs}},P_2}},  \ F\in\Rrs \otimes \Rrs.
\label{eq:realization4}
\end{split}
\end{equation}
\end{enumerate}
\end{corollary}

\begin{lemma}
\label{lemma:limFphi4}
We have
\begin{equation}
\lim_{P_1 \to \infty}
\frac{F(P_1, P_2)}{\phiH{g-1}(P_1)(x_1 - x_2)^2}
 = \phiH{g}(P_2),
\label{eq:limF4}
\end{equation}

and $\phiH{g}(P_2)$ is equal to
$x_{P_2}^{g_r} y_{{{\hr}},P_2}$ or 
$x_{P_2}^{g_s} y_{{{\hs}},P_2}$.

\end{lemma}

\begin{proof} Lemma \ref{lmm:phiH}
applied to $D_{r,s}:=D_{s,r}^{(+)}-D_{s,r}^{(-)}$
in Definition \ref{def:Drs} 
\end{proof}

\begin{remark} Corollary \ref{cor:Sigma} and 
Lemma \ref{lemma:limFphi4} hold for
 $r=0$, with
$$
\lim_{P_1 \to \infty}
\frac{F(P_1, P_2)}{\phi_{g-1}(P_1)(x_1 - x_2)^2}
 = \phi_{g}(P_2).
$$
\end{remark}

For $P=(x, y_{{\hr}}, y_{{\hs}})$,
we let $h_i(P)= 3 y_{{\hr}} y_{{\hs}} \nuII_{i}(P) / dx $,
$1\le i\le g$ (a representative in $R$ of the corresponding differential).

In Section \ref{transcendental} (Proposition \ref{useIntegrals}) we will use:
\begin{equation}
\label{eq:Omega_def4}
\begin{split}
\Omega^{P_1, P_2}_{Q_1, Q_2}
 &:= \int^{P_1}_{P_2} \int^{Q_1}_{Q_2} \Omega(P, Q) \\
 &= \int^{P_1}_{P_2} (\Sigma(P, Q_1) - \Sigma(P, Q_2))
 +\sum_{i = 1}^g \int^{P_1}_{P_2} \nuI_{i}(P)
\int^{Q_1}_{Q_2} \nuII_{i}(P).
\end{split}
\end{equation}
\section{Transcendental aspects}\label{transcendental}
In this section we set up the notation for the Jacobian of the curve
and its invariant vector fields, 
briefly recall properties of the sigma function, definitions for Wirtinger
strata and the Jacobi inversion problem. 

A pointed curve $(X,P)$ has a natural embedding in the normalized
Jacobian
$\nJJ:=\CC^g/\nGamma ,$ where $\nGamma$ is a normalized period lattice
${\ZZ}^g +\tau {\ZZ}^g$, and let
 $\nkappa:\CC^g \to \mathbb{C}/\nGamma$ be the projection.
 We refer to \cite[Ch. I]{fay}
for the definition of the Riemann theta function 
with characteristics $\theta\left[\begin{matrix}a
\\ b\end{matrix}\right] (z,\tau), \
z\in\mathbb{C}^g$.

Let $\tX$ be the fundamental covering space of $X$, 
$\varpi: \tX \to X$ the projection from the path space to the orbit space,
and $\iota: X \to \tX$
 the natural
embedding $X$ into $\tX$
such that $\varpi\circ \iota = id$.
For the $k$-symmetric product of $\tX$, $\cS^k \tX$,
we define the Abel map $\tv$ from $\cS^k \tX$ to $\CC^g$, normalized 
at $P\in X$, by
taking 
the sum of the integrals of a normalized basis of differentials of the first
kind,
 from the point $\iota P \in \tX$, to the $k$-tuple of points $P_j$, 
through any
 paths that join $\iota P$ to each $P_j$ (these become identified in the orbit
 space).
Using the embedding $\iota$, we also have a map $v$
from the symmetric product of the curve to $\nJJ$,
$v = \nkappa\circ \tv\circ \iota : \cS^k X \rightarrow \nJJ$.
We denote by $\nTheta$ the divisor of $\nJJ$ defined by $\theta(z,\tau)$ 
and recall the following classical result, for which we 
choose to quote Theorem 7 and Theorem 11 in \cite{Lew}:

\begin{proposition} \label{prop:lew}
\begin{enumerate}

\item
The  
 ``canonical'' theta divisor $\v(\cS^{k-1}X)$ 
is a translate of  $\nTheta$ by $\nkappa (\Delta)$, where $\Delta$ is
  the ``Riemann constant'' (cf. \cite[Ch. I, (13)]{fay}.

\item
An effective divisor $D$ of degree is $2g-2$
satisfies $\v(D)-\v((2g-2)P) +2  \Delta =0$  modulo $\nGamma$
if and only if $D$ is the divisor of a holomorphic differential.

\item When the canonical divisor equals $(2g-2)P$, 
the image $\nkappa (\Delta)$   of
Riemann constant is a point of order two on the Jacobian.
\end{enumerate}
\end{proposition}
We recall that $\cK_X = (2g-2)P$ exactly when the W-semigroup $H(X,P)$ is
symmetric. To streamline the theory, in \cite{KMP16} we introduced a positive
divisor $\fB$, of degree $d_0:=\mathrm{deg}(\fB)$, such that:

\begin{proposition}\label{prop:0}
For $\fB$ given by the equality (in the sense of linear equivalence)
$\cK_X=2D_0
= 2(g-1+d_0)P-2\fB,$ 
the  {\lq\lq}shifted Abel maps{\rq\rq} defined by
$
\tv_s(P_1, \ldots, P_k) = \tv(P_1, \ldots, P_k) + \tv(\iota\fB),
$
for $P_1, \cdots, P_k \in \tX$, and $\v_s := \nkappa \tv_s$,
and for
the  {\lq\lq}shifted Riemann constant{\rq\rq} defined by
$ 
\Delta_s:=\Delta -\tv(\iota\fB) \in \CC^g,
$
\begin{enumerate}
\item $\Delta_s$ belongs to $\frac{1}{2}\Gamma$,
\item
The difference between 
the ``shifted canonical theta divisor'' 
$\v_s(\cS^{g-1}X)$ and $\nTheta$
is given by the shifted Riemann constant $\Delta_s \in \CC^g$, i.e.,
as sets,
$$
\v_s(\cS^{g-1}X) +  \Delta_s = \nTheta 
\quad \mathrm{modulo}\ \nGamma.
$$
\end{enumerate}
\end{proposition}

We note that $\cK_X=2D_0$ defines $D_0$, or rather its image in the Jacobian,
which is a divisible Abelian group. We note also that Proposition
\ref{prop:0} is trivially true for $\fB = 0$ when $H(X,P)$ is symmetric.

\begin{corollary} \label{cor:thetadivisor}
 There is a theta characteristic,
\begin{equation}
   \delta:=\left[\begin{array}{cc}\delta'\ \\
       \delta''\end{array}\right]\in \left(\tfrac12\ZZ\right)^{2g},
  % \label{eq2.9} %3.15
\end{equation}
which is equal to the shifted Riemann constant $\Delta_s,$ 
namely, for every $(P_1, P_2, \ldots, P_{g-1}) 
\in \cS^{g-1} \tX$, $
\theta(\tv_s(P_1, \ldots, P_{g-1}) + \Delta_s, \tau) = 
\theta[\delta](\tv_s(P_1, \ldots, P_{g-1}), \tau) = 0.
$
\end{corollary}
\subsection{The trigonal cyclic case}
Recalling the calculation of $\cK_X$ in Subsection \ref{subsec:KX}, 
the divisor of Proposition
\ref{prop:0} is 
$
\fB= B_{s+1}+B_{s+2} +\cdots +B_{{r+s}},
$
thus
$
 \cK_{\Xrs} \sim 2 (g-1+r) \infty - 2\fB. 
$

We introduce  notation convenient for stating our results on the
sigma function, its vanishing order, and 
Jacobi inversion.
For a standard symplectic basis
$\alpha_{i}, \beta_{j}$  $ (1\leqq i, j\leqq g)$, 
of $H_1(X,\ZZ)$ % such that their intersection numbers are
 %$\alpha_{i}\cdot\alpha_{j}=\beta_{i}\cdot\beta_{j}= 0$ and
 %$\alpha_{i}\cdot\beta_{j}=\delta_{ij}$,
we denote the period matrices by
\begin{equation}
\begin{split}
   \left[\,\omegap{}  \ \omegapp{}  \right]&=
\frac{1}{2}\left[\int_{\alpha_{i}}\nuI_{j} \ \ \int_{\beta_{i}}\nuI_{j}
\right]_{i,j=1,2, \cdots, g},
   \label{eq:42.4}
\end{split}
  \end{equation}
so the
matrix $\tau = \omega^{\prime -1} \omega''$,
and  $\Gamma$ is the  lattice
generated by $\omegap{}$ and $\ \omegapp{}$, equivalently, by$(I_g, \tau)$.
The `unnormalized' Jacobian is given by $\cJ:=\CC^g/\Gamma$,
$\kappa : \CC^g \to \cJ$.
Since the basis $\{\nu^I_i\}\ (i=1,\ldots,g)$ differs from the standard basis of
$H^0(X,\Omega )$ normalized with respect to the $\alpha ,\beta $-cycles,
 we redefine the Abel map $\tw$ and attendant shifted Abel map 
$$
w_s:
\tX \rightarrow \CC^g, \quad\quad
\tw(P):=\int_\infty^P\nu^I\in\mathbb{C}^g,
$$
so that
$$
\tv=(2{\omega'})^{-1}\tw, \quad 
 \tv_s=(2{\omega'})^{-1}\tw_s,
\quad w = \kappa \circ \tw \circ \iota,\quad
 w_s = \kappa \circ \tw_s \circ \iota.
$$
Since $
\fB= B_{s+1}+B_{s+2} +\cdots +B_{{r+s}}$,
 the shifted Abel map is given by:
\begin{equation}
\tw_s(P_1, \ldots, P_k):=
\tw(P_1, \ldots, P_k)
+\tw(\iota B_{s+1}, \cdots , \iota B_{s+r}).
\label{eq:shiftedAmap}
\end{equation}
For each component $u_i$ of a vector $u\in \CC^g$,
we extend the weight wt that we introduced for $R_{\phi}$ and $\hR_{\phiH{}}$
and the  basis of holomorphic differentials that we introduced in
Proposition \ref{def:holo1-form},
$
\wt(u_i)=\Lambda_{g-i+1}+(g-i).
$
It should be noted that the image of the normalized Abel map does not
have such a natural weight, where the word ``natural'' refers to the order of
poles at infinity of functions on the affine part of the curve (and, in
consequence, to the cyclic action).
In particular, there is a $\mathbb{Z}/3\mathbb{Z}$-action on
$\cJ_g$, defined by:
 %in Subsection \ref{subsec:GmX},
$
\hzeta_3(x, y_{{\hr}}, y_{{\hs}})= (x, \zeta_3 y_{{\hr}}, \zeta_3^2
y_{{\hs}})
$ and  the holomorphic differentials
 in Definition \ref{def:holo1-form}.

We define the subvarieties $\WW^{k}$ and
$\WW_s^{k}$ 

\begin{equation}
   \WW^{k} :=  \abl(\cS^k \Xrs), \qquad
   \WW_s^{k} := \abl_s(\cS^k \Xrs), \qquad
\label{eq:W^k}
\end{equation}
We call them ``Wirtinger strata'' because their images under the $|2\Theta^\circ |$
divisor map are the classical Wirtinger varieties. By the Abel-Jacobi theorem,
$
   \WW^{g}  = \WW_s^g = \cJ.
$

To state vanishing theorems in Section \ref{results},
we also define the strata:
$\WW_{s,1}^k:=w(\cS^k_1 X)$
($\WW_1^k:=(\cS^k_1 X)$),
where 
$
\cS^n_m(X) := \{D \in \cS^n(X) \ | \
    \mathrm{dim} | D | \ge m\}.
$

\subsection{Affine functions on
$\Xrs$ and linear equivalence}\label{the mu function}

Let $n$ be a positive integer  and
$P_1, \ldots, P_n$ be in $\Xrs\backslash\infty$,
; define
%\begin{pmatrix}
%\phi_0(P_1) &\phi_1(P_1) & \phi_2(P_1)  
%& \cdots & \phi_{n-1}(P_1) \\
%\phi_0(P_2) & \phi_1(P_2) & \phi_2(P_2)
% & \cdots & \phi_{n-1}(P_2) \\
%\vdots & \vdots & \vdots & \ddots& \vdots \\
%\phi_0(P_n) & \phi_1(P_{n}) & \phi_2(P_{n}) 
% & \cdots&  \phi_{n-1}(P_{n})
%\end{pmatrix},
the Frobenius-Stickelberger (FS) matrices of $R$ and $\hR$,
\begin{equation}
\Psi_{n}(P_1, P_2, \ldots, P_n) :=(\phi_i(P_j))_{i=0,...,n-1,j=1,...,n},
\label{eq:partialInPsi4}
\end{equation}
$$
\hPsig{n}(P_1, P_2, \ldots, P_n) :=({\hat\phi}_i(P_j))_{i=0,...,n-1,j=1,...,n}.
$$
 %\begin{pmatrix}
%\phiH{0}(P_1) & \phiH{1}(P_1) & \phiH{2}(P_1) & \cdots & \phiH{n-1}(P_1) \\
%\phiH{0}(P_2) & \phiH{1}(P_2) & \phiH{2}(P_2) & \cdots & \phiH{n-1}(P_2) \\
%\phiH{0}(P_3) & \phiH{1}(P_3) & \phiH{2}(P_3) & \cdots & \phiH{n-1}(P_3) \\
%\vdots & \vdots  & \vdots & \ddots &\vdots \\
%\phiH{0}(P_{n}) & \phiH{1}(P_{n}) & \phiH{2}(P_{n}) & \cdots
%& \phiH{n-1}(P_{n})
%\end{pmatrix}.
The
{\it{Frobenius-Stickelberger (FS) determinant}} is
\begin{gather*}
\psig{n}(P_1, \ldots, P_n)
 := \det(\Psig{n}(P_1, \ldots, P_n)), \quad
\hpsig{n}(P_1, \ldots, P_n)
 := \det(\hPsig{n}(P_1, \ldots, P_n)).
\end{gather*}
%Here $\phi^{(g)}$ reads $\phiH{}$, $\phiM{}$ or $\phit{}$.
%We call this matrix {\it{Frobenius-Stickelberger (FS) matrix}}

\begin{definition} \label{def:mul}
For $P, P_1, \ldots, P_n$ $\in (\Xrs\backslash\infty) \times 
\cS^n(\Xrs\backslash\infty)$,
we define $\mu_{n}(P)$ by
$$
\mu_{n}(P): =
\mu_{n}(P; P_1,  \ldots, P_n): =
\lim_{P_i' \to P_i}\frac{1}{\psig{n}(P_1' , \ldots, P_n' )}
\psig{n+1}(P_1' , \ldots, P_n' , P),
$$
$$
\muH{n}(P): =
\muH{n}(P; P_1,  \ldots, P_n): =
\lim_{P_i' \to P_i}\frac{1}{\hpsig{n}(P_1' , \ldots, P_n' )}
\hpsig{n+1}(P_1' , \ldots, P_n' , P),
$$
where the $P_i^\prime$ are generic,
the limit is taken (irrespective of the order) for each $i$;
and $\mu_{n, k}(P_1, \ldots, P_n)$ by
and $\muH{n, k}(P_1, \ldots, P_n)$ by
$$
\mu_{n}(P)
 = \phi_{n}(P) +
\sum_{k=0}^{n-1} (-1)^{n-k}\mu_{n, k}(P_1, \ldots, P_n) \phi_{k}(P),
$$
$$
\muH{n}(P)
 = \phiH{n}(P) +
\sum_{k=0}^{n-1} (-1)^{n-k}\muH{n, k}(P_1, \ldots, P_n) \phiH{k}(P),
$$
with the convention $\mu_{n, n}(P_1, \ldots, P_n) \equiv $
$\muH{n, n}(P_1, \ldots, P_n) \equiv 1$.
\end{definition}

\begin{remark}
When $r=0$, $\muH{n}(P)$ is equal to $y_{\hr}(P)\mu_{n}(P)$
because of Remark \ref{rmk:3.8}.
\end{remark}

The meromorphic functions we introduced 
enable us to express the addition structure
of $\mathrm{Pic}\ \Xrs$ in terms of  FS-matrices.
We developed this theory for $(n,s)$ curves and here we adapt it to curves
$(X,P)$ that have a non-symmetric semigroup.

For $n$ points $(P_i)_{i=1, \ldots, n}$ $\in \Xrs\backslash\infty$,
we find an element of $R$ associated with
 any point $P= (x,y)$ in $(\Xrs\backslash\infty )$,
$\alpha_n(P) :=\alpha_n(P; P_1, \ldots, P_n)
= \sum_{i=0}^{n} a_i \phi_i(P)$, $a_i \in \CC$ and $a_n = 1$,
which has a  zero at each point $P_i$ (with multiplicity, if
the $P_i$ are repeated)
and has smallest possible order of pole at $\infty$ with this property.
Then
$\alpha_n$ can be identified with $\mu_{n}(P)$.
We have the following lemma.
\begin{lemma}\label{lmm:2theta2}
Let $n$ be a positive integer.
For $(P_i)_{i=1,\ldots, n}\in \cS^n(X \backslash\infty) $,
the  function  $\alpha_n$
over $\Xrs$ induces the map 
$
\alpha_n:
\cS^n(\Xrs \backslash\infty)  \to \cS^{N(n) - n}(\Xrs),
$
sending
$(P_i)_{i=1,\ldots, n}$ $\in \cS^n(\Xrs \backslash\infty) $
to an % unique
 element $(Q_i)_{i=1,\ldots, N(n)-n}\in \cS^{N(n)-n}(\Xrs)$,
such that
$$
\sum_{i=1}^{n} P_i - n \infty
\sim - \sum_{i=1}^{N(n)-n} Q_i  + (N(n)-n) \infty .
$$
\end{lemma}

Note that the function
$\displaystyle{\widehat \alpha_n:=\frac{\muH{n}dx}{3y_{{\hr}}y_{{\hs}}}}$ 
is non-singular over
$X\setminus \infty$, unlike
$\displaystyle{\frac{\mu_{n}dx}{3y_{{\hr}}y_{{\hs}}}}$.
The divisor of $\muH{n}(x)$ contains $\sum_{i=1}^{s+r} B_a - (s+r) \infty\sim$
$2(B_{s+1}+\cdots + B_{s+r}) - 2r \infty$ for any  $n$.
The following is proved as in \cite[Lemma 6.2]{KMP13}.

\begin{lemma}\label{lem:2theta4}
Let $n$ be a positive integer.
For $(P_i)_{i=1,\ldots, n}\in \cS^n(\Xrs \backslash\infty) $,
the  function  $\alphaH_n$
over $\Xrs$ induces a map 
$\displaystyle{
\alphaH_n:
\cS^n(\Xrs \backslash\infty)  \to 
\cS^{\hN(n) - n-g-1}(\Xrs),
}$ that sends
$(P_i)_{i=1,\ldots, n}$
 %$\in \cS^n(\Xrs \backslash\infty)$
 to a
 $(Q_i)_{i=1,\ldots,\hN(n)-n-g-1}$ such that
$$
\sum_{i=1}^{n} P_i +B_{s+1}+\cdots + B_{s+r} - (n+r) \infty
$$
$$
\sim
\displaystyle{
 -\left( \sum_{i=1}^{\hN(n)-n-g-1} Q_i
+B_{s+1}+\cdots + B_{s+r} - (\hN(n)-n-s) \infty \right).}
$$
\end{lemma}

%\begin{proof}
%The linear equivalence holds because:
%\begin{gather*}
%\begin{split}
%&\sum_{i=1}^{n} P_i +\sum_{a=1}^{s+r} B_a
%+ \sum_{i=1}^{\hN(n)-n-g-1} Q_i  -\hN(n) \infty\\
%&\sim\sum_{i=1}^{n} P_i  + \sum_{i=1}^{\hN(n)-n-g-1} Q_i
%+2(B_{s+1} + \cdots + B_{s+r}) -(\hN(n) +r-s)\infty \sim 0.
%\end{split}
%\end{gather*}
%\end{proof}

\bigskip
In order for the preimage of 
$\alpha_n$
and $\widehat\alpha_n$ to
include the base point $\infty$, we extend the maps as follows:
for an effective divisor $D$ in $\cS^n(\Xrs)$ of degree $n$,
let $D'$  be the maximal subdivisor  of $D$ which does
not contain $\infty$,
$D = D' + (n-m) \infty $
where $\deg D'=m (\le n)$ and
$D' \in \cS^m(\Xrs\backslash\infty)$,
and define $\overline{\alpha}_n$
by  $\overline{\alpha}_n(D)={\alpha}_m(D^\prime )+[N(n)-n-
(N(m)-m)]\infty.$
We modify the derivation of the Abel-Jacobi theorem and Serre duality
given in \cite{KMP13}, based on Remark \ref{rmk:CanonDiv}
($\alpha_n$ and $\widehat\alpha_n$ behave in analogous ways to each other).
By linear equivalence (cf. Lemmas \ref{lmm:2theta2}
 and \ref{lem:2theta4}):
\begin{proposition} \label{prop:addition}
For a positive integer $n$, the
Abel map composed with $\alpha_n$ (given in terms of $\mu_{n}$) and
$\alphaH_n$ (in terms of $\muH{n}$) induce
$$
\iota_n :  \WW^n \to  \WW^{\NM(n) - n-g-1}, \qquad
\hiota_n :  \WW_s^n \to  \WW_s^{\hN(n) - n-g-1}.
$$
\end{proposition}

\begin{remark} \label{rmk:addition}
{\rm{
As per Proposition \ref{prop:addition}, under the shifted Abel
map 
we have,  for any $P_1, P_2, \cdots, P_{g}$ and appropriate $Q$'s in
$\Xrs$,
\begin{eqnarray}
-\abl_s(P_1, P_2, \cdots, P_{g-1})
&=\abl_s(Q_1', Q_2', \cdots, Q_{g-1}') , \mbox{ mod }\Gamma,\nonumber \\
-\abl_s(P_1, P_2, P_3, \cdots, P_{g}) &=\abl_s(Q_1, Q_2, Q_3, \cdots, Q_{g})
\mbox{ mod }\Gamma .
\end{eqnarray}
The first relation,
\begin{equation}
\begin{split}
-\abl(P_1, P_2, \cdots, P_{g-1}) &=\abl(Q_1, Q_2, \cdots, Q_{g-1}) 
+ 2\abl(B_{s+1}, \cdots + B_{s+r}),
\label{eq:-1u}
\end{split}
\end{equation}
 shows %(see \cite[I]{Mu}  p.166)
that Serre duality on $\Xrs$ is given as,
 $\hiota_{{g-1}} : \WW_s^{g-1} \to  \WW_s^{{g-1}}$ by
\begin{gather*}
\begin{split}
&P_1 + P_2 + \cdots + P_{g-1} +B_{s+1}+\cdots + B_{s+r}
- 
(2r+s) \infty\\
&\sim -\left( Q_1 + Q_2 + \cdots + Q_{g-1} +B_{s+1}+\cdots
+ B_{s+r} - 
(2r+s)\infty \right).
\end{split}
\end{gather*}
}}
\end{remark}

We therefore denote
 $\mathrm{image}(\hiota_n)$  by $[-1]_s\WW_s^n$,
especially $\hiota_{g}: \WW_s^{g} \to [-1]_s \WW_s^{g}$.
For $n\ge g$,
$\hiota_g \circ\hiota_n$ gives addition in the Picard group,
$$
 \WW_s^n \stackrel{\hiota_n}{\longrightarrow}
  \WW_s^{g} \stackrel{\hiota_g}{\longrightarrow} \WW_s^{g},
\quad
(\abl_s(P_1,\ldots ,P_n) \equiv \abl_s( Q_1,\ldots ,Q_g)\
\mbox{ mod  } \Gammars.
$$
In particular,
the addition law on the Jacobian is given by
 $\hiota_g \circ\hiota_{2g}$
$$
 \WW_s^{2g} \stackrel{\hiota_{2g}}{\longrightarrow}
  \WW_s^{g}
  \stackrel{\hiota_g}{\longrightarrow} \WW_s^{g},
\quad
(\abl_s(P_1, \ldots ,P_g , P'_1, \ldots , P'_g)\equiv \abl_s(Q_1, \ldots , Q_g)
\ \mathrm{ mod  }\Gammars.
$$

The above arguments
and Lemma \ref{lmm:Ng-1} give
 the following corollary
(Serre duality and the Abel-Jacobi theorem):

\begin{corollary}
$
-\WW_s^{g-1} =\WW_s^{g-1}, \quad
-\WW_s^{g} =\WW_s^{g}.
$
\end{corollary}

It is the `shift' that allows us to conclude:
\begin{proposition} \label{prop:ws_vanishes}
For some $(P_1, \ldots, P_{g-1}) \in \cS^{g-1}X$,
 $\abl_s(P_1, \ldots, P_{g-1}) = 0$.
\end{proposition}
\begin{proof}
Noting $g=r+s-1$, we set 
$(P_1, \ldots, P_{g-1}) = (B_1, \ldots, B_{s+r})$; then,
$$
(y_\hr)=B_1 + \cdots + B_s +2(B_{s+1}, \ldots, B_{s+r}) 
-(2r+s) \infty\sim 0.
$$
\end{proof}

\subsection{Vector fields}
We give expressions for
differential operators on $\cS^k X$
or invariant vector fields on the Jacobian
using the coordinates 
of the Abel map.
We use the convention that for $P_a \in \Xrs$,
$P_a$ is expressed by $(x_a, y_{{{\hr}},a}, y_{{{\hs}},a})$
or $(x_{P_a}, y_{{{\hr}},{P_a}}, y_{{{\hs}},{P_a}})$.
By letting $(u_1,...,u_g) :=\abl(P_1, \cdots, P_{g})$, we have
\begin{equation}
\begin{pmatrix}
\partial/\partial{u_1}\\
\partial/\partial{u_2}\\
\partial/\partial{u_3}\\
\vdots\\
\partial/\partial{u_{g}}
\end{pmatrix}
=
{\hPsig{g}}^{-1}(P_1, P_2, \ldots, P_g)
\begin{pmatrix}
3y_{{{\hr}},1}y_{{{\hs}},1} \partial/\partial{x_1}\\
3y_{{{\hr}},2}y_{{{\hs}},2} \partial/\partial{x_2}\\
3y_{{{\hr}},3}y_{{{\hs}},3} \partial/\partial{x_3}\\
\vdots\\
3y_{{{\hr}},g}y_{{{\hs}},g} \partial/\partial{x_{g}}
\end{pmatrix}.
\end{equation}
 %In other words,
%$$
%\sum_{i=1} \epsilon_i\frac{\partial}{\partial u_i}
%=
%|{\hPsig{g}}^{-1}(P_1,\ldots,P_g)|
%\left|
%\begin{matrix}
%\phiH{0}(P_1) & \phiH{1}(P_1) & \cdots & \phiH{g-1}(P_1) &
%3y_{{{\hr}},1}y_{{{\hs}},1} \partial/\partial{x_1}\\
%\phiH{0}(P_2) & \phiH{1}(P_2) & \cdots & \phiH{g-1}(P_2) &
%3y_{{{\hr}},2}y_{{{\hs}},2} \partial/\partial{x_2}\\
%%\phiH{0}(P_3) & \phiH{1}(P_3) & \cdots & \phiH{g-1}(P_3) &
%%3y_{{{\hr}},3}y_{{{\hs}},3} \partial/\partial{x_3}\\
%\vdots & \vdots & \ddots & \vdots & \vdots \\
%\phiH{0}(P_{g}) & \phiH{1}(P_{g}) & \cdots  & \phiH{g-1}(P_{g}) &
%3y_{{{\hr}},4}y_{{{\hs}},4} \partial/\partial{x_{g}}\\
%\epsilon_1 & \epsilon_2 & \cdots & \epsilon_{g} &
%\end{matrix}
%\right|.
%$$
 These relations hold for the image of 
the shifted Abel map, thus
 similar relations involving submatrices of $\widehat\Psi_{g}$ hold
over the strata
$\abl_s(\cS^k \Xrs)$ for $k<g$;
also:
\begin{equation}
          \sum_{i, j=1}^{g} \phiH{i-1}(P_1) \phiH{j-1}(P_2)
	\frac{\partial^2 }
{\partial \abl_i(P_1)\partial \abl_j(P_2)}
= 9 y_{{{\hr}},1} y_{{{\hs}},1} y_{{{\hr}},2} y_{{{\hs}},2}
	\frac{\partial^2 }{\partial x_1\partial x_2}.
\label{eq:partial_H4}
\end{equation}

\subsection{The sigma function}\label{sigma}
In genus one, the Weierstrass sigma function \cite[21$\cdot$ 43]{ww}
$$
\sigma (u) =\frac{2\omega_1}{\theta^\prime}
\exp\left\{\frac{\eta_1}{2\omega_1}u^2\right\}
\theta_1\left(\frac{u}{2\omega_1}|\frac{\omega_2}{\omega_1}\right),
$$
where $\theta_1$ is the theta function with characteristics
$\displaystyle\left[\begin{matrix} 1/2\\ 1/2\end{matrix}\right]$,
$\theta_1^\prime:=\theta_1^\prime(0|\frac{\omega_2}{\omega_1}))$ 
is ``equivalent'' in the sense of \cite[Ch. VI]{lang} 
to a first-order theta function with characteristics,
and the complete integrals of first and second kind, $\omega_1,\omega_2$
and $\eta_1,\eta_2$ satisfy the Legendre relation: 
$\eta_1\omega_2-\eta_2\omega_1={\pi\imath /2}.$
Theta functions satisfy a given relation with respect to the period lattice, 
they are called equivalent when they differ by a ``trivial'' theta function,
and \cite[Ch. X Th. 1.1]{lang} shows that there is a unique normalized entire
theta function representing  on $\mathbb{C}^g$ 
the inverse image under $\kappa$ of a positive
divisor on $\mathbb{C}^g/\Gamma$. Unlike theta,
 $\sigma$ has
modular invariance (under a different choice of basis of $\Gamma$),  up
to a root of unity. To generalize Weierstrass' sigma to higher genus,
the authors of \cite{bel, ble} 
use the ``fundamental 2-differential of the second kind'' \cite[2.2.3]{bel},
to write a suitable basis of $H^1(X,\mathbb{C})$ which satisfies the
``generalized Legendre relations''.
We obtained explicit expressions for the trigonal cyclic curve $(X,P)$
and wrote sigma in \cite{KMP13}, for a (3,7,8) curve; for the general trigonal
cyclic case, the proofs follow the same lines as those in 
\cite[Section 5]{KMP13}, and we omit them.

We write the periods:
\begin{equation}
\begin{split}
   \left[\,\etap{}  \ \etapp{}  \right]&:=
\frac{1}{2}\left[\int_{\alpha_{i}}\nuII_{j} \ \
                 \int_{\beta_{i}}\nuII_{j}
\right]_{i,j=1,2,\cdots, g}.
   \label{eq2.5}
\end{split}
  \end{equation}

 %Let $\tau_{Q_1, Q_2}$ be the 
%differential of the third kind that
%has residues $+1$, $-1$ at ${Q_1, Q_2}$,
%is regular everywhere else,  and is  normalized,
% $\int_{\alpha_i} \tau_{P, Q} = 0$ for every $i$.
%The following Lemma corresponds
% to  Corollary 2.6 (ii) in \cite{fay}:
%\begin{lemma} \label{lemma:4.1}
%$$
%{\Omega}^{P_1, P_2}_{Q_1, Q_2} =
%\int^{P_1}_{P_2}
%\tau_{Q_1, Q_2} + \sum_{i, j = 1}^g \gamma_{ij}
%\int^{P_1}_{P_2} \nuI_{i} \int^{Q_1}_{Q_2} \nuI_{j},
%$$
%where $\gamma{} = {\omegap{}}^{-1} \etap{}$ has entries $\gamma_{ij}$.
%\end{lemma}

\begin{proposition}
The following matrix satisfies the  {\it generalized Legendre relation}:
\begin{equation}
   M := \left[\begin{array}{cc}2\omegap{} & 2\omegapp{} \\
               2\etap{} & 2\etapp{}
     \end{array}\right], \quad
  M\left[\begin{array}{cc} & -I_g \\ I_g & \end{array}\right]{}^t {M}
  =2\pi\sqrt{-1}\left[\begin{array}{cc} & -I_g \\ I_g &
    \end{array}\right],
\end{equation}
where $I_g$ is the $g\times g$ unit matrix.
\end{proposition}
 %Theta characteristics  $u'$ and $u''$ in $\RR^g$ 
%correspond to a vector $u\in\CC^g$,
%\begin{equation*}
%   u=2\omegap{} u'+2\omegapp{} u''.
%\end{equation*}
 Using  theta characteristics $\delta$
as in Corollary \ref{cor:thetadivisor}, 
we define  $\sigma$ as an entire function of (a column-vector)
$u={}^t\negthinspace (u_1, u_2, \ldots, u_g)\in \mathbb{C}^g$,

\begin{equation}
\begin{aligned}
   \sigma{}(u) %&=\sigma{}(u;M)=\sigma{}(u_1, u_2, \ldots, u_g;M) \\
   &=c\,\text{exp}(-\tfrac{1}{2}{}\ ^t\negthinspace
u\etap{}{\omegap{}}^{-1}u)
   \vartheta\negthinspace
   \left[\delta\right](\frac{1}{2}{\omegap{}}^{-1} u;\
{\omegap{}}^{-1}\omegapp{}) \\
   &=c\,\text{exp}(-\tfrac{1}{2}\ ^t\negthinspace
u\etap{}{\omegap{}}^{-1}\  u) \\
   &\hskip 20pt\times
   \sum_{n \in \ZZ^g} \exp \big[\pi \sqrt{-1}\big\{
    \ ^t\negthinspace (n+\delta'')
      {\omegap{}}^{-1}\omegapp{}(n+\delta'')
   + \ ^t\negthinspace (n+\delta'')
      ({\omegap{}}^{-1} u+2\delta')\big\}\big],
\end{aligned}
   \label{de_sigma}
\end{equation}
where  $c$ is a certain constant, in fact a rational
function of the $b$'s in the equation of the curve (Subsection
\ref{singular}). This function has the required modular invariance and
 is homogeneous with respect to the extended weight of
$\wt_\lambda$.

\begin{proposition} \label{prop:pperiod}
For $u$, $v\in\CC^g$, and $\ell$
$(=2\omegap{}\ell'+2\omegapp{}\ell'')$ $\in\Gammars)$, if we define
$
  L(u,v)    :=2\ {}^t{u}(\etap{}v'+\etapp{}v''),\
  \chi(\ell):=\exp[\pi\sqrt{-1}\big(2({}^t {\ell'}\delta''-{}^t
  {\ell''}\delta') +{}^t {\ell'}\ell''\big)],$
the following holds
\begin{equation}
	\sigma{}(u + \ell) =
\sigma{}(u) \exp(L(u+\frac{1}{2}\ell, \ell)) \chi(\ell).
        \label{eq:4.11}
\end{equation}
\end{proposition}
\subsection{The Riemann fundamental relation}
Using (\ref{eq:-1u}) to detect the divisors corresponding to
the minus-sign operation on $\cJ_{g}$,
we review a relation which we  call the Riemann fundamental
relation \cite[\S196]{baker1897}):
\begin{proposition}\label{useIntegrals}
For $(P, Q, P_i, P'_i) \in \tX^2 \times (\cS^g(\tX)\setminus \cS^g_1(\tX))
 \times (\cS^g(\tX)\setminus \cS^g_1(\tX))$,
$
        u  := \Abl_s(P_1, \ldots, P_g),   \quad
        u'  := \Abl_s(P'_1, \ldots, P'_g),   \quad
$
$$
\exp\left(
\sum_{i, j = 1}^g
   {\Omega}_{\varpi P_i, \varpi P'_j}^{\varpi P, \varpi Q} \right)
=
\frac{\sigma{}(\Abl(P) - u) \sigma{}(\Abl(Q) - u')}
     {\sigma{}(\Abl(Q) - u) \sigma{}(\Abl(P) - u')}$$
 %&=\frac{\sigma{}(\Abl(P) - \Abl_s(P_1, \ldots, P_g))
 %       \sigma{}(\Abl(Q) - \Abl_s(P'_1, \ldots, P'_g))}
 %    {\sigma{}((\Abl(Q) - \Abl_s(P_1, \ldots, P_g))
 %     \sigma{}(\Abl(P) - \Abl_s(P'_1, \ldots, P'_g))}.
 %\end{align*}
\end{proposition}
\begin{proposition} \label{prop:wpxx=F}
For $(P, P_1, \ldots, P_g) \in X \times \cS^g(X) \setminus \cS^g_1(X)$
and $u \in \kappa^{-1} \abl_s(P_1, \ldots, P_g)$,
the equality,
$$
        \sum_{i, j = 1}^g \wp_{i, j}
          \left(\tw(\iota P) - u\right)
         \phiH{i-1}(P)
         \phiH{j-1}(P_a) =\frac{F(P, P_a)}{(x-x_a)^2},
$$
holds for every $a = 1, 2, \ldots, g$,
where we set %$v\in \kappa^{-1}u \subset \CC^g$,
$$
        \wp_{ij}(u) := -\frac{\sigma_{i}(u) \sigma_{j}(u)
              -  \sigma{}(u) \sigma_{ij}(u)}
               {\sigma{}(u)^2}\equiv
        -\frac{\partial^2}{\partial u_i\partial u_j}\log\sigma{}(u),
$$
$$
\sigma_{i_1,i_2,\ldots,i_n}(u)
:=
\prod_{j=1}^n \frac{\partial}{\partial u_{i_j}}\sigma{}(u).
$$
\end{proposition}

\begin{proof}
First we note that the above $\wp$ functions are defined over
the Jacobian $\cJ$, namely
$\wp_{ij}(u) = \wp_{ij}(u + \ell)$ for $\ell \in \Gamma$, because of the
functional relation satisfied by $\sigma$.
Using the property of the vector fields in (\ref{eq:partial_H4}) and
taking logarithm of both sides in the Riemann
fundamental relation and differentiating
along $P_1=P$ and $P_2=P_a$, we obtain the claim.
\end{proof}

By translation-invariance of
the $\wp$-function, we may view
the domain as subset of  $\JJ$; more precisely,
$\wp_{ij}(u) := \wp_{ij}(u')$ for $u' \in \CC^g$ and $u := \kappa u
\in \JJ$.

\section{Vanishing on Wirtinger strata and 
Jacobi inversion}\label{results}

Using the subvarieties (\ref{eq:W^k}), we let 
\begin{equation}
	\Theta_s^{k} := \WW_s^{k} \cup [-1] \WW_s^{k},
\label{eq:Theta:k}
\end{equation}
and we define
\begin{equation}
	\Theta^{k}_{s,1} := w_s(\cS^{k}_1 (X)) \cup [-1] w_s(\cS_1^{k}(X)).
\label{eq:Theta:k1}
\end{equation}

For a Young diagram $\lambda=(\lambda_1, \ldots, \lambda_n)$,
the Schur function
$s_\lambda$ are defined by the ratio of determinants 
of $n\times n$ matrices \cite{Mac},
$$
	s_{\lambda}(T) = \frac{|t_{i}^{\lambda_j+n-j}|}{|t_i^{j-1}|}
$$
where $t=^t(t_1,\ldots, t_n)$ are the transpose of the rows.
We also regard it as a function of
$\displaystyle{
          T=^t(T_1,\ldots,T_n), \quad
T_k :=\frac{1}{k}\sum_{i=1}^g t_j^k,
}$
 denoted by
$
	S_{\lambda}(T) = s_{\lambda}(t).
$

Following a formula proven for the  rational/polynomial case  \cite{bel2},
Nakayashiki showed the following for the case of an $(n,s)$ curve
\cite{Na} 
and a general Riemann surface 
\cite[Theorem 10]{N3}:

\begin{proposition} \label{prop:Schur0}
With the notation of
 Proposition \ref{prop:lew}, for a general non-singular curve
$X$, we have the following results:
For 
$e +\Delta\in (\nkappa )^{-1}\nTheta$ 
and for $u \in \CC^g$,
\begin{equation}
C\theta\left( (2\omega^{\prime})^{-1} u+ e,\tau\right) = 
S_\Lambda(T)|_{T_{\Lambda_i+g-i}=u_i}+\mbox{higher-weight terms},
\end{equation}
where $C$ is a suitable constant.
\end{proposition}

\bigskip

We note that for  $e$ in Proposition \ref{prop:Schur0},
$2\omega'e +\tw(\iota\fB) \in \kappa^{-1}\Theta_s^{g-1}$ and thus
$$
C\theta \left[\begin{matrix}\delta'\\ \delta''\end{matrix}
\right] \left( (2\omega^{\prime})^{-1} u+ e
+\tv(\iota\fB), \tau\right) = 
S_\Lambda(T)|_{T_{\Lambda_i+g-i}=u_i}+\mbox{higher-weight terms},
$$
for $\delta$ in the definition of $\sigma$ (\ref{de_sigma})
from Corollary \ref{cor:thetadivisor}.
We also note that $u+2\omega' e+ \tw(\iota\fB) \in
\tw_s(\cS^{g}\tX)=\CC^g$. 
Further since $\kappa ( 2\omega' e)$ belongs to the canonical theta divisor
$w(\cS^{g-1}\tX)$,  
by letting $e = 
-\tv(\iota(\fB))$ from Proposition \ref{prop:ws_vanishes}
(with $v(B_1, \ldots, B_{s+r}) = -v(\fB)$ and $r+s = g-1$),
we have the following result:
\begin{proposition} \label{prop:Schur}
The leading term in the Taylor expansion
of the $\sigma$ function associated with $\Xrs$, with normalized
constant factor $c$,
 is expressed by a Schur function  
$$
   \sigma{}(u) = S_{{\Lambda}}(T)|_{T_{\Lambda_i + g - i} = u_i}
            + \sum_{\alpha} a_\alpha u^\alpha,
$$
where $a_\alpha\in \QQ[b_1, \cdots, b_{s+r}]$,
$\alpha = (\alpha_1,..., \alpha_g)$
and  $u^\alpha = u_1^{\alpha_1} \cdots u_g^{\alpha_g}$.
\end{proposition}

\begin{remark}
In \cite{MP14}, we investigated the Riemann-Kempf theory for
a $C_{rs}$ curve (another name, coding-theory terminology, for an $(n,s)$ curve), obtaining
the expansion of sigma in terms of the Schur functions for
Young diagrams of symmetric semigroups. However, the 
arguments we gave never used 
the symmetry of the Young diagrams.
Thus the results are applicable to the non-symmetric case, and
can be obtained from \cite{MP14} directly, alternative to 
Nakayashiki's results \cite{N3}.
\end{remark}

We introduce the truncated Young diagrams of $\Lambda
=(\Lambda_1, \Lambda_2, \ldots,
\Lambda_k, \ldots, \Lambda_g)$:
$\Lambda^{(k)}
=(\Lambda_1, \Lambda_2, \ldots, \Lambda_k)$
and 
$\Lambda^{[k]}
=(\Lambda_{k+1}, \ldots, \Lambda_g)$.

For a given truncated Young diagram $\Lambda^{[k]}$ (viewed as a partition
of the total number of its boxes), with rank $n_k$ (the length of the
diagonal), 
its  ``Frobenius characteristics'' \cite[\S4.1]{FH} is
$(a_1, a_2, a_3, ..., a_{n_k};$ $
b_1, b_2, b_3, ..., b_{n_k})$,  $a_i\ge a_j$ and $b_i\ge b_j$ for $i<j$,
with $a_i$ and $b_i$ the number of boxes below and to the 
right of the $i$-th box of the diagonal.
 Define:
$\displaystyle{
N_k := \sum_{1}^{n_k}(a_i + b_i +1).
}$

\begin{proposition}
Let us consider a truncated Young diagram $\Lambda^{[k]}$
with the Frobenius characteristics of the partition
$(a_1, a_2, a_3, ..., a_{n_k};
b_1, b_2, b_3, ..., b_{n_k})$. For each pair 
$(a_i, b_i)$, there is $\ell_i \in \{k+1, \ldots, g\}$,
with
$\ell_1 = k+1$  for $i=1$, 
 such that
$$
       \Lambda_{\ell_i}+g-\ell_i = a_i + b_i +1.
$$
\end{proposition}

Let us denote such $\ell_i$ by $L^{[k]}(a_i,b_i)$ and
$$
\sharp_{k}:=\{ L^{[k]}(a_1,b_1), L^{[k]}(a_2,b_2),\ldots,
L^{[k]}(a_{n_k},b_{n_k}) \}.
$$
and for $i \le k + 1$,
$$
\sharp_{k}^{(i)}:=\{ i, L^{[k]}(a_2,b_2),\ldots,
L^{[k]}(a_{n_k},b_{n_k}) \}.
$$ 
Note that $\sharp_{k}= \sharp_{k}^{(k+1)}$.
We state Riemann's singularity theorem %(cf. \cite[VI.1]{ACGH})
 and its version in \cite[Th. 2]{N3}
as follows:
%with the usual notation of $h^i$ for the dimension of the 
%cohomology space $H^i$: \label{pg:cohomology}
\begin{proposition} \label{prop:RST} 
Let $(P_1, \ldots, P_k)$ belong to $\cS^k(X\backslash\infty) \setminus 
(\cS^k_1(X)\cap \cS^k(X\backslash\infty)) $ and
$u \in \kappa^{-1} \abl_s(P_1, \ldots, P_k)$.
\begin{enumerate}
\item
 For every multi-index  $(\alpha_1, \ldots, \alpha_m)$ 
with $\alpha_i\in \{ 1, \ldots, g\}$ and $m < n_k$,
$$
	\frac{\partial^m}
        {\partial u_{\alpha_1} \ldots \partial u_{\alpha_m}}
            \sigma(u) = 0.
$$
\item
For
$(\beta_1, \ldots, \beta_{n_k})=\sharp_k^{(i)}$ $(i=1,2,\cdots,k+1)$, 
\begin{equation}
	\frac{\partial^{n_k}}
         {{\partial u}_{\beta_1} \ldots \partial u_{\beta_{n_k}}}
            \sigma(u) \neq 0.
\label{eq:Rvn-nvn}
\end{equation}
\end{enumerate}
\end{proposition}

\begin{proposition} \label{cor:Fay}
For all $1\le k\le g-1$
 $u^{[k]} \in \kappa^{-1}(\Theta_s^k\setminus
\left(\Theta^k_{s,1}\cup \Theta_s^{k-1}\right))$,
$u \in \CC^g$,
$v\in \WW^{1}$
$\displaystyle{
        \left.\frac{\partial^{\ell}}
         {\partial u_g^{\ell}} \sigma(u)
         \right|_{u = u^{[k]}} = 0, \  \ell < N_k; \quad
        \left. \frac{\partial^{N_k}}
         {\partial {u_g}^{N_k}}
          \sigma(u)\right|_{u = u^{[k]}} \neq 0.
}$
\end{proposition}

%\subsection{Jacobi inversion formula}
\subsection{Jacobi inversion formulae over $\Theta^k$}

For any  cyclic trigonal curve, of $(3,{\hr},{\hs})$-type 
 (including $r=0$),
we have the Jacobi inversion formulae:

%We introduce
%$$
%S^n_m(X) := \{D \in S^n(X) \ | \
%    \mathrm{dim} | D | \ge m\},
%$$
%where $|D|$ is the complete linear system
%$\Abl^{-1}(\Abl(D))$ in IV.1 of
%\cite[IV.1]{ACGH}.

\begin{theorem} {\bf{(Jacobi inversion formula)}}\label{thm:8.1}
\begin{enumerate}
For $(P, P_1, \cdots, P_{g}) \in \Xrs \times
\cS^g(\Xrs) \setminus \cS^g_1(\Xrs)$,
we have
\begin{enumerate}
\item
$\muH{g}(P; P_1, \ldots, P_{g}) = \phiH{g}(P)-$
$\sum_{j=1}^{g} \wp_{g j}$ $(\abl_s(P_1,$ $\ldots ,P_{g}))$ $
\phiH{j - 1}(P)$.

\item
$
\wp_{g, k+1}(\abl_s(P_1,\ldots ,P_{g})) = (-1)^{g-k}\muH{g, k}
(P_1, \ldots, P_{g}),$
$(k=0, \ldots, g-1)$.
\end{enumerate}

\end{enumerate}
\end{theorem}

 \begin{proof}
Same as in Prop. 4.6 of 
 %\cite[Prop. 4.6]{MP08}.
\cite{MP08}.
\end{proof}

%\label{Jacobi inversion formulae over $\Theta^k$}

%We  introduce the shifted theta characteristics
%\begin{equation}
%	\Theta_s^{k} := \WW_s^{k} \cup [-1] \WW_s^{k}.
%\label{eq:Theta:k}
%\end{equation}

Bearing in mind Proposition \ref{prop:RST},
we have the following theorem, in which both sides
are obtained by taking a limit by $P_{k+1} \to \infty$ as
mentioned in \cite[Th. 5.1]{MP08}; indeed, some of the left-hand
sides  are given by zero over zero
but they are well-defined in the limit.

\begin{theorem} {\bf{(Vanishing Theorems)}}\label{thm:8.2}
The following relations hold for the 
$\muH{}$ functions:
\begin{enumerate}

\item $\Theta_s^g$ case:
for
$(P_1, \ldots, P_g) \in \cS^g(\Xrs) \setminus \cS^g_1(\Xrs)$ and
%$u = \pm \abl_s(P_1, \ldots, P_g)\in \kappa^{-1}(\Theta^g)$,
$u \in \kappa^{-1}
 \abl_s(P_1, \ldots, P_g)\subset \kappa^{-1}(\Theta_s^g)=\CC^g$,
\begin{gather*}
\frac{\sigma_{i}(u) \sigma_{g}(u) - \sigma_{g i}(u) \sigma{}(u)}
               {\sigma^2(u)}
         =(-1)^{s+r-i} \muH{g, i - 1}(P_1, \ldots, P_g),
         \quad \mbox{for } 1\le i \le g.\\
\end{gather*}

 %\item $\Theta_s^{g-1}$ case:
%for
%$(P_1, \dots, P_{g-1}) \in \cS^{g-1}(\Xrs) \setminus \cS^{g-1}_1(\Xrs)$ and
%$u \in \kappa^{-1} \abl_s(P_1, \ldots, P_{g-1})
%\subset \kappa^{-1}(\Theta_s^{g-1})$,
%\begin{gather*}
%\frac{\sigma_{i}(u)}
%               {\sigma_{g}(u)}=
%\left\{\begin{matrix}
%         (-1)^{g-i} \muH{g-1, i-1}(P_1, \ldots, P_{g-1})
%                 & \mbox{for } 1\le i < g,\\
%          1 & \mbox{for } i = g.
%       \end{matrix}\right.
%\end{gather*}

\item $\Theta_s^{k}$ case:
for
$(P_1, \ldots, P_k) \in \cS^k(\Xrs) \setminus \cS^k_1(\Xrs)$ and
$u \in \kappa^{-1} \abl_s(P_1, \ldots, P_k)\subset \kappa^{-1}(\Theta_s^k)$,
$(k=1,2,\ldots, g-1)$,
\begin{gather*}
\frac{\sigma_{i}(u)}
               {\sigma_{k+1}(u)}=
\frac{\sigma_{\sharp_k^{(i)}}(u)}
               {\sigma_{\sharp_k}(u)}=
\left\{\begin{matrix}
          (-1)^{k-i+1} \muH{k, i-1} (P_1, \ldots, P_k)
          & \mbox{for }  0 < i \le k,\\
          1 & \mbox{for } i = k+1,\\
         0 &\mbox{for } k+ 1 < i \le  g.
       \end{matrix}\right.
\end{gather*}

\item $\Theta_s^{1}$ case:
$(P_1) \in \cS^1(\Xrs)$ and 
$u \in \kappa^{-1} \abl_s(P_1)\subset \kappa^{-1}(\Theta_s^k)$,
\begin{gather*}
\frac{\sigma_{1}(u)}
               {\sigma_{2}(u)}
=\frac{\sigma_{\sharp_1^{(1)}}(u)}
               {\sigma_{\sharp_1}(u)}= \frac{\phiH{1}}{\phiH{0}}.
\end{gather*}
and if $r<s-3$, the right hand side is equal to $x(P_1)$.
\end{enumerate}
\end{theorem}

\begin{proof}
Essentially the same as in Th. 5.1 of 
\cite{MP14}.
%\cite[Th. 5.1]{MP14}.
\end{proof}

\begin{remark}
Every 
 curve in Weierstrass normal form with W-semigroup
$\langle3, \hr,\hs\rangle$ corresponds to the same monomial
ring as the cyclic case (cf. Proposition \ref{prop:Z4}). 
Though the structure ring $R$ differs,
it is expected that the same $R_\phi$ and $\widehat R_{\hphi}$, 
 bases and subsets, also
 play  similar roles as in the cyclic case.
Therefore, we expect to be able to adapt 
 the results of this paper   to any
 curve in Weierstrass normal form 
 with W-semigroup $\langle3, \hr,\hs\rangle$.
\end{remark}

\setcounter{section}{0}
\renewcommand{\thesection}{\Alph{section}}
\section{Appendix: Proof of Proposition 4.5}
A computation shows:
\begin{lemma} \label{lmm:h(s,t)}
The function
$h(t,s):=\dfrac{t^\ell - s^\ell}{t-s} - \dfrac{d }{dt} t^\ell$
satisfies the following relations:
\begin{enumerate}
\item $h(t,t) = 0$,

\item $\dfrac{h(t,s)}{t-s} \in \QQ[t,s]$, and

\item
$
\dfrac{h(t,s)}{t-s} = -\sum_{a=0}^{\ell-2} (a+1) s^{\ell-a-2} s^{a}
\mbox{ for } \ell>1,\quad
\dfrac{h(t,s)}{t-s} = 0
\mbox{ for } \ell=0.1.\quad
$
\end{enumerate}
\end{lemma}

 %\begin{proof}
%(1) and (2) are obvious. For $a+b = \ell-1$,
%we have
%$$
%\dfrac{t^a s^b - t^{a+b}}{t-s} =
%t^a(s^{b-1}+s^{b-2} t + \cdots + t^{b-1})
%$$
%so that
%$$
%\dfrac{h(t,s)}{t-s} = t^{\ell-2} + 2t^{\ell-3} s
%+ 3 t^{\ell-4} s^2 + \cdots + (\ell-2) s^{\ell-2}.
%$$
%\end{proof}

We now compute the differentials of $y_{\hr}$ and $y_{\hs}$:
$$
\frac{d}{dx} y_{\hs}=\frac{1}{3 y_{\hs}^2}
\left(2k_r k_r' k_s + k_r^2 k_s'\right), \quad
\frac{d}{dx} y_{\hr}=\frac{1}{3 y_{\hr}^2}
\left(2k_s k_s' k_r + k_s^2 k_r'\right), \quad
$$
where we set $k_{a,P} = k_a(x_P)$ and $k_{a,P}' = d k_a(x_P)/d x_P$.
 %We consider the differential:
%\begin{equation*}
%\begin{split}
%&   \frac{\partial }{ \partial x_Q}
%   \frac{ y_{{{\hr}},P} y_{{{\hs}},P} +y_{{{\hr}},P} y_{{{\hs}},Q}
%+y_{{{\hr}},Q} y_{{{\hs}},P}}
% {3(x_P - x_Q)  y_{{{\hr}},P}  y_{{{\hs}},P} } d x_P \\
%&=\frac{1}{9(x_P - x_Q)
%y_{{{\hr}},P} y_{{{\hs}},P} y_{{{\hr}},Q} y_{{{\hs}},Q}}
%\Bigr[
%\frac{
%3(y_{{{\hr}},P} y_{{{\hs}},P} +y_{{{\hr}},P} y_{{{\hs}},Q}
% +y_{{{\hr}},Q} y_{{{\hs}},P})
%y_{{{\hr}},Q} y_{{{\hs}},Q}}{(x_P - x_Q) }\\
%& \quad\quad
%+
%\Bigr(y_{{{\hr}},P}
%\frac{y_{{{\hr}},Q}}{y_{{{\hs}},Q}}\left(
%2k_{s,Q} k_{s,Q}' k_{r,Q} +k_{s, Q}^2  k_{r,Q}'\right)
%+ y_{{{\hs}},P}
%\frac{y_{{{\hs}},Q}}{y_{{{\hr}},Q}}\left(
%2k_{s,Q} k_{r,Q}k_{r,Q}'+ k_{s, Q}' k_{r,Q}^2)\right)
%\Bigr) \Bigr].\\
%\end{split}
%\end{equation*}
 The result follows from the equalities:
$$
   \frac{\partial }{ \partial x_Q}
   \frac{ y_{{{\hr}},P} y_{{{\hs}},P} +y_{{{\hr}},P} y_{{{\hs}},Q} +y_{{{\hr}},Q} y_{{{\hs}},P}}
{3(x_P - x_Q)  y_{{{\hr}},P}  y_{{{\hs}},P} }
 -
   \frac{\partial }{ \partial x_P}
   \frac{ y_{{{\hr}},Q} y_{{{\hs}},Q}
+y_{{{\hr}},Q} y_{{{\hs}},P} +y_{{{\hr}},P} y_{{{\hs}},Q}}
{3(x_Q - x_P)  y_{{{\hr}},Q}  y_{{{\hs}},Q} }
$$
 %\begin{equation*}
%\begin{split}
%&=\frac{1}{9(x_P - x_Q)
%y_{{{\hr}},P} y_{{{\hs}},P} y_{{{\hr}},Q} y_{{{\hs}},Q}}
%\Bigr[
%\frac{
%3 (y_{{{\hr}},P} y_{{{\hs}},Q} +y_{{{\hr}},Q} y_{{{\hs}},P})
%(y_{{{\hr}},Q} y_{{{\hs}},Q} -y_{{{\hr}},P} y_{{{\hs}},P}) }{(x_P - x_Q) }
%\\
%& \quad\quad
%+ \Bigr(
%y_{{{\hr}},P} y_{{{\hs}},Q}(
%2k_{{s},Q}' k_{{r},Q} +k_{{s}, Q}  k_{{r},Q}'
%+2k_{{s},P} k_{{r},P}'+ k_{{s}, P}' k_{{r},P})\\
%& \quad\quad\quad
%+y_{{{\hr}},Q} y_{{{\hs}},P}(
%2k_{{s},P}' k_{{r},P} +k_{{s}, P}  k_{{r},P}'
%+2k_{{s},Q} k_{{r},Q}'+ k_{{s},Q}' k_{{r},Q})
%\Bigr) \Bigr]\\
%&=
%\frac{1}{9
%y_{{{\hr}},P} y_{{{\hs}},P} y_{{{\hr}},Q} y_{{{\hs}},Q}}
%\Bigr[
%\\
%& \quad\quad
%+
%\frac{
%y_{{{\hr}},P} y_{{{\hs}},Q}}{
%(x_P-x_Q)}
%\left(
%k_{{s},Q}' k_{{r},Q} +k_{s+r, Q}'
%+k_{{s},P} k_{{r},P}'+ k_{s+r, P}'
%-3 \frac{k_{s+r,P}-k_{s+r,Q}}{x_P-x_Q}
%\right)\\
%& \quad\quad
%-\frac{
%y_{{{\hr}},Q} y_{{{\hs}},P}}{
%(x_Q-x_P)}
%\left(
%k_{{s},P}' k_{{r},P} +k_{s+r,P}'
%+k_{{s},Q}k_{{r},Q}'+ k_{s+r,Q}'-
%3 \frac{k_{s+r,Q}-k_{s+r,P}}{x_Q-x_P}
%\right)
%\Bigr) \Bigr].\\
%\end{split}
%\end{equation*}
 We show that the above is equal to
$$
\frac{1}{9
y_{{{\hr}},P} y_{{{\hs}},P} y_{{{\hr}},Q} y_{{{\hs}},Q}}
\left(
y_{{{\hr}},P} y_{{{\hs}},Q} D_{r,s}(P, Q)
- y_{{{\hr}},Q} y_{{{\hs}},P} D_{r,s}(Q, P)\right),
$$
where $D_{r,s}(P, Q)= D^{(+)}_{ s,r}(P, Q) - D^{(-)}_{s,r}(Q, P)$
from   Definition \ref{def:Drs}.
Indeed, by Lemma \ref{lmm:h(s,t)},
$$
 \frac{1}{x_Q-x_P}\left(
 \frac{k_{s+r,Q}-k_{s+r,P}}{x_Q-x_P}
- k_{s+r,Q}'\right) = -
 \sum_{j=0}^{s+r-2}
 \sum_{i=0}^{s+r-j-2}
(i+1)\lmsr{j}x_P^{r+s-j-i-2} x_Q^{i},
$$
$$
 \frac{1}{x_Q-x_P}\left(
 \frac{k_{s+r,Q}-k_{s+r,P}}{x_Q-x_P}
- k_{s+r,P}'\right) =
 \sum_{j=0}^{s+r-2}
 \sum_{i=0}^{s+r-j-2}
(i+1)\lmsr{j}x_Q^{r+s-j-i-2} x_P^{i},
$$
\begin{gather*}
\begin{split}
 \frac{1}{x_Q-x_P}\left(
 \frac{k_{s+r,Q}-k_{s+r,P}}{x_Q-x_P}
-k_{{s},P}' k_{{r},P} -k_{{s},Q}k_{{r},Q}'\right)
&= \sum_{j=0}^{r-2}
 \sum_{i=0}^{r-j-2}
 \sum_{k=0}^{s}
(i+1)\lmr{j}\lms{k} x_P^{r-j-i-2} x_Q^{i+k}\\
&-\sum_{j=0}^{s-2}
 \sum_{i=0}^{s-j-2}
 \sum_{k=0}^{r}
(i+1)\lms{j}\lmr{k} x_Q^{s-j-i-2} x_P^{i+k},\\
\end{split}
\end{gather*}
noting that
$
 k_{s+r,Q}-k_{s+r,P}
 =k_{s,Q}(k_{r,Q} -k_{r,P})
 +k_{r,P}(k_{s,Q} -k_{r,P}).
$
Thus we obtained the expression of $D_{r,s}$.

%    Bibliographies can be prepared with BibTeX using amsplain,
%    amsalpha, or (for "historical" overviews) natbib style.
%\bibliographystyle{amsplain}
%    Insert the bibliography data here.


\begin{thebibliography}{ABCDE}

\bibitem[Ac]{accola}
 %Accola, Robert D. M. Vanishing thetanulls for some dihedral and cyclic
 %coverings of Riemann surfaces. Kodai Math. J. 28 (2005), no. 1, 73-91.  
 R.D.M. Accola, \textit{On cyclic trigonal Riemann
 surfaces. I}, Trans. Amer. Math. Soc. 283 (1984), no. 2, 423-449. 

%\bibitem[ACFH]{ACGH}
% \by{E. Arbarello, M. Cornalba, P.A. Griffiths and J. Harris}
 % \book{Geometry of Algebraic Curves. Vol. I}
 % \publ{Springer-Verlag} New York, 1985.



\bibitem[Ay1]{ayano}
T. Ayano, \paper{Sigma functions for telescopic curves},
    Osaka J. Math.  \vol{51} (2014), 459-481.

\bibitem[Ay2]{ayano1}
T. Ayano, \paper{
On Jacobi Inversion Formulae for Telescopic Curves}
    SIGMA   \vol{12} (2016), Paper No. 086, 21 pp. 


\bibitem[B1]{baker1897}
  \by{H.F. Baker}
  \book {Abelian functions. Abel's theorem
and the allied theory of theta functions}
  \publ{Cambridge University Press, Cambridge, 1995}
   Reprint of the 1897 original.

\bibitem[B2]{baker1907} 
\by{H.F. Baker}
 {An introduction to the theory of
    multiply-periodic functions}, { University Press XVI },
  Cambridge, 1907.


\bibitem[B-A]{a-b}
\by{Maria Bras-Amor\'{o}s}
\paper{Numerical Semigroups and Codes}
Chapter 5 of Algebraic Geometry Modeling in Information Theory,
E. Martinez-Moro (ed.),World Scientific, 2013.



%\bibitem[BEL]{bel}
%V.M. Buchstaber,  V.Z. Enolski\u{\i} and D.V. Le\u{\i}kin, 
%\textit{Hyperelliptic Kleinian functions and applications}. Solitons, geometry, and
%topology: on the crossroad, 1--33,
	%		       Amer. Math. Soc. Transl. Ser. 2, 179,
	%		       Adv. Math. Sci., 33, Amer. Math. Soc.,
	%		       Providence, RI, 1997.  



\bibitem[BEL1]{bel}
\by{V.M. Buchstaber,  V.Z. Enolski\u{\i} and D.V. Le\u{\i}kin} 
\paper{Kleinian functions, hyperelliptic Jacobians and applications}
\jour{Reviews in Mathematics and Mathematical Physics}
\vol{10} (1997)  1-103.

\bibitem[BLE]{ble} 
\by{V.M. Buchstaber, D.V. Le\u{\i}kin and V.Z. Enolski\u{\i}}
\paper{$\sigma$-functions of $(n,s)$-curves}
\jour{Uspekhi Mat. Nauk} \vol{54} (1999),
  no. 3 (327), 155--156; 
translation in Russian Math. Surveys 54 (1999), no. 3,
  628-629. 


\bibitem[BEL2]{bel2}
  \by{V.M. Bukhshtaber, V.Z. \`Enol'ski\u{\i} and D.V. Le\u{\i}kin}
  \paper{Rational analogues of Abelian functions}
  \jour{Funct. Anal. Appl.} \vol{33} \yr{1999} \pages{83-94}.


%\bibitem[DC]{DC}
%\by{A. Del Centina}
%\paper{Weierstrass points and their impact in the study of algebraic
%curves: a historical account from the {\lq\lq}L\"uckensatz{\rq\rq}
%to the 1970s}
%\jour{Ann. Univ. Ferrara} \vol{54} \yr{2008} \pages{37-59}.

\bibitem[EEL]{EEL}
   \by{J.C. Eilbeck, V.Z. Enolskii and D.V. Leykin}
   \paper{On the Kleinian construction of abelian functions of canonical
   algebraic curves}
 SIDE III---symmetries and integrability of difference equations (Sabaudia,
   1998),  121--138, CRM Proc. Lecture Notes, 25, Amer. Math. Soc.,
   Providence, RI, 2000.

\bibitem[EH]{EH}
\by{D. Eisenbud and J. Harris}
\paper{Existence, decomposition, and limits of certain Weierstrass points}
\jour{Invent. Math.} \vol{87} \yr{1987} \pages{495-515}. 




\bibitem[FZ]{farkaszemel}
\by{H. M. Farkas, S. Zemel}
\book{Generalizations of Thomae's Formula for $Z_n$ Curves}
Springer, 2011.


\bibitem[F]{fay} 
\by{J.~D. Fay}
 {Theta functions on {R}iemann surfaces},
  Lectures Notes in Mathematics, vol. {\bf 352}, Springer, Berlin,
  1973.

\bibitem[FH]{FH}
  \by{W. Fulton and J. Harris}
  \book{Representation Theory}
Graduate Texts in Mathematics, Vol. 129. Springer-Verlag, Berlin-New York,
1991.

\bibitem[He]{herzog}
  \by{J. Herzog}
  \paper{Generators and relations of Abelian semigroup and semigroup ring}
\jour{Manuscripta Math.} \vol{3} \yr{1970} \pages{175-193}.


\bibitem[Ka]{kato}
\by{T. Kato}
\paper{Weierstrass normal form of a Riemann surface and its applications}
\jour{S\^{u}gaku} 32 (1980), no. 1, 73-75. 
 
\bibitem[Kl1]{klein1}
\by{F. Klein}
\paper{Ueber hyperelliptische sigmafunctionen}
\jour{Math. Ann.} \vol{27} (3) (1886)


\bibitem[Kl2]{K}
\by{F. Klein}
\paper{Ueber hyperelliptische sigmafunctionen
(Zweite Abhandlung)}
\jour{Math. Ann.} \vol{32} (3) (1888)
351-380

\bibitem[Kl3]{klein2}
\by{F. Klein}
\paper{Zur theorie der Abel'schen functionen}
\jour{Math. Ann.} \vol{36} (1) (1890) 1-83.


\bibitem[KMP1]{KMP13} 
\by{J. Komeda, S. Matsutani and E. Previato}
\paper{The sigma function for Weierstrass semigroups
$\langle3,7,8\rangle$
and $\langle6,13,14,15,16\rangle$}
\jour{Internat. J. Math.} \vol{24} \yr{2013} no. 11, 1350085, 58 pp.
 


\bibitem[KMP2]{KMP16} 
\by{J. Komeda, S. Matsutani and E. Previato}
\paper{The Riemann constant for a non-symmetric 
Weierstrass semigroup}
Arch. Math. (Basel) \vol{107} (2016), no. 5, 499-509. 

\bibitem[KS]{korotkin}
\by{D. Korotkin and V. Shramchenko}
\textit{On higher genus Weierstrass sigma-function}.
{Physica D} \vol{241} ({2012}), {2086--2094}.


\bibitem[La]{lang}
\by{S. Lang}
\book{Introduction to algebraic and abelian functions,
Second edition}
Graduate Texts in Mathematics, 89.
Springer-Verlag, New York-Berlin, 1982.

\bibitem[Le]{Lew}
\by{J. Lewittes}
\paper{Riemann surfaces and the theta functions}
\jour{Acta Math.} \vol{111} \yr{1964} 37-61.

\bibitem[Mac]{Mac}
\by{I.G. Macdonald}
\book{Symmetric functions and Hall polynomials}
 Oxford Mathematical Monographs (2nd ed.) 1985.

%\bibitem[Mat]{Mat}
%\by{H. Matsumura}
%\book{Commutative ring theory, CSAM 8},
%translated by M. Reid,
%Cambridge, Cambridge, 1986.




\bibitem[MK]{komedamatsutani}
\by{S. Matsutani and J. Komeda}
 \paper{Sigma functions for a space curve of type
(3,4,5)} J. Geom. Symmetry Phys. \vol{30} (2013), 75-91. 


\bibitem[MP1]{MP08}
\by{ S. Matsutani and E. Previato}
\paper{Jacobi inversion on strata of the Jacobian of the $C_{rs}$ curve
$y^r=f(x)$ I}
\jour{J. Math. Soc. Japan} \vol{60} \yr{2008} \pages{1009-1044}.

\bibitem[MP2]{MP14}
\by{ S. Matsutani and E. Previato}
\paper{Jacobi inversion on strata of the Jacobian of the $C_{rs}$ curve
$y^r=f(x)$ II}
\jour{J. Math. Soc. Japan} 
 \vol{66} \yr{2014} \pages{647-692}.

\bibitem[Miu]{Miu}
\by{S. Miura}
 \paper{Linear codes on affine algebraic curves, IEICE Trans.  }
\vol{J81-A} \yr{1998} \pages{1398-421}, in Japanese.

\bibitem[N1]{Na}
\by{A. Nakayashiki}
\paper{On algebraic expansions of sigma functions for $(n, s)$ curves}
\jour{Asian J. Math.} \vol{14} \yr{2010} \pages{ 175-212}.

\bibitem[N3]{N3}
\by{A. Nakayashiki}
\paper{Tau function approach to theta functions}
 Int. Math. Res. Not. IMRN (2016) 2016 (17): 5202-5248.

\bibitem[Pi]{P}
  \by{H. C. Pinkham}
  \paper{Deformation of algebraic varieties with $G_m$ action}
\jour{Ast\'erisque} \vol{20}  \yr{1974} \pages{1-131}.



\bibitem[R]{ramalfons}
\by{J. Ram\'{\i}rez Alfons\'{\i}n}
The Diophantine Frobenius problem, Oxford Univ. Press, 2005.

\bibitem[WW]{ww} 
\by{E.T. Whittaker and G.N. Watson}
\book{A Course of
      Modern Analysis,} {Cambridge University Press} 1927.

%\bibitem[W]{W} \by{K. Weierstrass} 
%\paper{Uber Normalformen algebraischer Gebilde}
% Mathematische Werke III, 297-307,
%Georg Ohms and Johnson reprint (1967).

\end{thebibliography}
\end{document}